\documentclass[a4paper]{amsart}

\usepackage{amssymb}
\usepackage{amsmath,amsthm}
\usepackage{enumerate,paralist}
\usepackage{amstext}
\usepackage{psfrag}
\usepackage{dsfont}
\usepackage[dvips]{epsfig}
\usepackage[english]{babel}
\usepackage{color}

\usepackage{mathtools}
\usepackage{latexsym}
\usepackage{mathrsfs}
\usepackage{MnSymbol}
\usepackage{bbm}

\theoremstyle{plain}
\newtheorem{theorem}{Theorem}
\newtheorem{corollary}{Corollary}
\newtheorem{lemma}{Lemma}
\newtheorem{proposition}{Proposition}
\theoremstyle{definition}

\newtheorem{example}{Example}
\newtheorem{remark}{Remark}
\newtheorem{assumption}{Assumption}

\numberwithin{theorem}{section}
\numberwithin{corollary}{section}
\numberwithin{lemma}{section}
\numberwithin{definition}{section}
\numberwithin{example}{section}
\numberwithin{remark}{section}
\numberwithin{proposition}{section}
\numberwithin{assumption}{section}

\usepackage{color}
\definecolor{CB}{rgb}{0.1,0.5,0.1}
\definecolor{YK}{rgb}{0.1,0.2,0.7}
\definecolor{H}{rgb}{0.7,0.1,0.2}
\definecolor{MY}{rgb}{0.5,0,0.45}

\renewcommand{\Re}{\ensuremath{{\rm Re\,}}}

\renewcommand{\leq}{\leqslant}
\renewcommand{\geq}{\geqslant}

\newcommand{\E}{\mathbb{E}}

\newcommand{\1}{\mathds 1}

\newcommand{\ffi}{\varphi}
\newcommand{\eps}{\varepsilon}
\newcommand{\BB}{\mathscr{B}}

\newcommand{\Id}{\mathop{\mathrm{Id}}\nolimits}

\renewcommand{\Re}{\ensuremath{{\rm Re\,}}}

\renewcommand{\leq}{\leqslant}
\renewcommand{\geq}{\geqslant}

\newcommand{\FF}{\mathcal F}

\newcommand{\PP}{\mathbb{P}}
\newcommand{\aA}{\mathcal{A}}

\newcommand{\calc}{\mathscr{C}}
\newcommand{\Ham}{\mathcal{H}}

\def\Nat{\mathbb{N}}

\def\liml{\lim\limits}

\def\Dom{\text{\rm Dom}}

\def\ffi{\varphi}
\def\eps{\varepsilon}

\def\Om{\Omega}
\def\cR{{\mathbb R}}
\def\cRd{{\mathbb R}^d}
\def\cC{{\mathbb C}}

\def\pd{\partial}

\begin{document}

\title[Subordination principle and Feynman-Kac formulae]{Subordination principle and Feynman-Kac formulae for generalized time-fractional evolution equations}

\author{Christian Bender}
\address{Universit\"{a}t des Saarlandes, Fachrichtung Mathematik, Postfach 15 11 50, 66041 Saarbr\"{u}cken.} 
\email{\texttt{bender@math.uni-sb.de}}

\author{Marie  Bormann}
\address{ Technische Universit\"{a}t Braunschweig, Institut f\"{u}r Mathematische Stochastik,    Universit\"{a}tsplatz 2, 38106 Braunschweig.}
\email{\texttt{m.bormann@tu-bs.de}, \texttt{marie-christin.bormann@math.uni-leipzig.de }} 

\author{Yana A. Butko}
\address{ Technische Universit\"{a}t Braunschweig,  Institut f\"{u}r Mathematische Stochastik,  Universit\"{a}tsplatz 2, 38106 Braunschweig.}
\email{\texttt{y.kinderknecht@tu-bs.de}, \texttt{yanabutko@yandex.ru}}

\date{\today}

\begin{abstract}

We consider   generalized time-fractional evolution equations of the form $$u(t)=u_0+\int_0^tk(t,s)Lu(s)ds$$ with a fairly general memory kernel $k$ and an operator $L$ being the generator of a strongly continuous semigroup. In particular,    $L$ may be the generator $L_0$ of a Markov process $\xi$ on some state space $Q$, or $L:=L_0+b\nabla+V$ for a suitable potential $V$ and drift $b$. Moreover,  $L$ may be the generator of a  subordinate semigroup or a Schr\"{o}dinger type group. This class of evolution equations includes in particular time- and space- fractional heat and Schr\"odinger type equations.
We show that a subordination principle holds for such evolution equations and  obtain Feynman-Kac formulae for solutions of these  equations with the use of different stochastic processes, such as subordinate Markov processes and  randomly scaled Gaussian processes.  In particular, we obtain some Feynman-Kac formulae with generalized grey Brownian motion and other related self-similar processes with stationary increments.

\bigskip

   
   \bigskip
   
   
\end{abstract}

\maketitle

\section{Introduction}\label{Sec:Intro}

Many natural phenomena exhibit a diffusive behaviour such that the displacement distribution has a non-Gaussian form and / or its variance is not linear in time. Such phenomena are usually called \emph{anomalous diffusion} and are observed in many complex systems, ranging from turbulence and plasma physics to soft matter and neuro-physiological systems (see, e.g.,~\cite{C4CP03465A,MR1809268,MR3916448} and references therein). Many different models have been proposed for the description of such phenomena. One of the earliest approaches obtains different regimes of anomalous diffusion as  proper scaling limits of  continuous time random walks.   Stochastic processes which arise as such scaling limits are \emph{Markov processes time-changed by}   so-called \emph{inverse subordinators} (see, e.g.,~\cite{MR1874479,MR2766141,MR3987876,MR2074812,MR2442372} and references therein).      In the frame of this approach, time- and / or space-fractional evolution equations emerge as governing equations for the underlying stochastic processes. The basic evolution equation in this context is the time- and / or space-fractional heat equation which replaces the standard heat equation, the basic equation of the classical diffusion models:
\begin{align}\label{eq:time+space-fracEvEq}
u(t,x)=u_0(x)-\int_0^t\frac{(t-s)^{\beta-1}}{\Gamma(\beta)}\left(-\frac12\Delta  \right)^\gamma u(s,x)ds,\qquad \beta\in(0,1],\quad\gamma\in(0,1].
\end{align}
Here $-\left(-\frac12\Delta  \right)^\gamma$, $\gamma\in(0,1)$, is the fractional Laplacian~\cite{MR3613319}. Equation~\eqref{eq:time+space-fracEvEq} serves as a governing equation for the process $(Y_{\mathcal{E}^\beta_t})_{t\geq0}$ which is a symmetric $2\gamma$-stable L\'evy process $(Y_t)_{t\geq0}$ time-changed by an independent inverse $\beta$-stable subordinator $(\mathcal{E}^\beta_t)_{t\geq0}$.

Another direction in the theoretical description of anomalous diffusion emerges by modelling the diffusion in complex media and interprets the anomalous character of the diffusion as a consequence  of a very heterogeneous character of the  environment \cite{MR3280006,PhysRevLett.113.098302,MR3903618,Jain2017,MR3916448,Sposini_2018}.  Some of the models   of diffusion in complex media are based on \emph{randomly scaled Gaussian processes} (RSGP), see, e.g.,~\cite{MR3586912,MR3513003,Wang_2020} and references therein. In particular,   processes of the form $(\sqrt{A}G_t)_{t\geq0}$, where $(G_t)_{t\geq0}$ is a Gaussian process and $A$ is a nonnegative random variable, which is independend of $(G_t)_{t\geq0}$, are considered. The most well-known RSGP of this type is the \emph{generalized grey Brownian motion} (GGBM) $(X^{\alpha,\beta}_t)_{t\geq0}$, $\alpha\in(0,2)$, $\beta\in(0,1]$. The GGBM $(X^{\alpha,\beta}_t)_{t\geq0}$ was introduced in works of Mainardi, Mura and their coauthors  \cite{MR2501791,MR2430462,MR2588003}, and can be realized as 
\begin{align}\label{eq:GGBM}
X^{\alpha,\beta}_t:=\sqrt{A_\beta}B^{\alpha/2}_t,
\end{align}   
 where $B^{\alpha/2}_t$ is a fractional Brownian motion (FBM) with Hurst parameter $\alpha/2$ and $A_\beta$ is a   nonnegative random variable with $\E[e^{-\lambda A_\beta}]=E_\beta(-\lambda)$ (here $E_\beta$ is the Mittag-Leffler function with parameter $\beta$). Due to the properties of fractional Brownian motion, the GGBM (and some further related RSGP) are \emph{self-similar processes with  stationary increments} (SSSI), which makes such processes attractive for modelling. Note however, that fractional Brownian motion (and hence GGBM) is neither a Markov process,  nor a semimartingale. The governing equation of GGBM is the following time-stretched time-fractional heat equation
\begin{equation}\label{GGBM-eq}
u(t,x)=u_0(x)+\frac{\alpha}{\beta\Gamma(\beta)}\int_0^t s^{\frac{\alpha}{\beta}-1}\left(t^{\frac{\alpha}{\beta}}-s^{\frac{\alpha}{\beta}}  \right)^{\beta-1}\frac12\Delta u(s,x)ds.
\end{equation}
Equation~\eqref{GGBM-eq} reduces to the time-fractional heat equation (i.e. equation~\eqref{eq:time+space-fracEvEq} with $\gamma:=1$) if $\alpha:=\beta$. 
For $\gamma\in(0,1)$, the time- and space-fractional heat equation~\eqref{eq:time+space-fracEvEq} is shown to be the governing equation for another SSSI RSGP (see~\cite{MR3513003}).  Therefore, both classes of stochastic processes  discussed above can be used to solve the same time- (and  space-) fractional evolution equations. These classes of processes have however very different nature and properties (see, e.g.,~\cite{doi:10.1063/5.0029913}).
Let us finally mention another type  of RSGP considered in the literature. It is represented by the  \emph{scaled Brownian motion with random  diffusivity} (see, e.g.,~\cite{DOSSANTOS2021110634} and references therein) which can be thought of as a solution $(X_t)_{t\geq0}$ of a heuristic stochastic equation $X_t=\int_0^t\dot{B}_s\sqrt{A_s \tau(s)}ds$, where $(\dot{B}_t)_{t\geq0}$ is a white noise, $(A_t)_{t\geq0}$ is a suitable nonnegative stochastic process  which is independent of Brownian motion $({B}_t)_{t\geq0}$ and $\tau\,:\,[0,\infty)\to[0,\infty)$ is a deterministic function (usually a power function). The processes $(B_{At^\theta})_{t\geq0}$ and $(\sqrt{A} B_{t^\theta})_{t\geq0}$  may be considered as special cases of scaled Brownian motion with random  diffusivity (cf.~\cite{DOSSANTOS2021110634}).

In this paper, we study a general class of evolution equations 
 of the form
\begin{align}\label{eq:EvEq-general}
u(t)&=u_0+\int_0^t k(t,s)Lu(s)ds,\qquad t>0, 
\end{align}
where $k$ is a fairly general memory kernel and $L$ is the infinitesimal generator of a strongly continuous semigroup $(T_t)_{t\geq 0}$ acting on some Banach space $X$. We identify conditions on the memory kernel $k$ which admit to write the solution operator of this equation 
 in the form
$$
\Dom(L)\rightarrow X,\quad   u_0\mapsto \int_0^\infty (T_au_0)\, \mathcal{P}_{A(t)}(da)
$$
for a family $(\mathcal{P}_{A(t)})_{t\geq 0}$ of probability measures on the positive real line, which depends on  $k$ only. We, thus, consider this representation as a subordination principle associated to the memory kernel $k$. 
 We state the subordination principle in Section~\ref{sec:main}, and in particular discuss how to obtain stochastic representations of the solution, if the $L$ is (a Bernstein function of) the infinitesimal generator of a Markov process (plus a potential).  The most natural stochastic reprensentations of such an approach are given in terms of  time-changed Markov processes.  In Section~\ref{sec:Examples}, we explain, however, how to arrive at representations in terms of non-Markovian processes such as generalized grey Brownian motion or even in terms of stochastic differential equations driven by more general randomly scaled fractional Brownian motions. In Section~\ref{sec:Examples}, we discuss also  evolution equations~\eqref{eq:EvEq-general} in the special case of kernels $k$ of convolution type; we present some examples of such kernels and discuss an equivalent form of equation~\eqref{eq:EvEq-general} in the setting of generalized Caputo type fractional derivatives.   Finally, the proofs are provided in Section \ref{sec:proofs}. While the main results can be considered as generalizations of our previous results in \cite{bebu1} beyond the case of pseudo-differential operators $L$ associated to L\'evy processes, the proofs are completely different, relating (an approximate version of) the subordination principle to a family of Volterra equations  via the Hille-Phillips functional calculus.


\section{Main results}\label{sec:main}
In this section, we state our main results.
We start from a general abstract setting and extract some Feynman-Kac formulae as special cases afterwards. 

\begin{assumption}\label{ass:L}
Let  $X$ be a Banach space with a norm $\|\cdot\|_X$. Let $(T_t)_{t\geq0}$ be a strongly continuous 
 semigroup on $X$ 
  with generator 
   $(L,\Dom(L))$.  
\end{assumption}
We consider the  evolution equation~\eqref{eq:EvEq-general} with operator $L$ as in Assumption~\ref{ass:L}, 
with $u_0\in\Dom(L)$, $u\,:\,[0,\infty)\to X$ and $k$ satisfying the following Assumptions~\ref{ass:k}--\ref{ass:Phi}. 

\begin{assumption}\label{ass:k} We consider a Borel-measurable kernel $k\,:\,(0,\infty)\times(0,\infty)\to\cR$ satisfying the condition:  $\exists\,\alpha^*\in[0,1)$ and $\exists\,\eps>0$ such that for each $T>0$ 
\begin{align*}
K_T:=\sup\limits_{0<t\leq T}t^{\alpha^*-\frac{1}{1+\eps}}\|k(t,\cdot)\|_{L^{1+\eps}((0,t))}<\infty.
\end{align*}
\end{assumption}
In order to identify the family of probability measures $(\mathcal{P}_{A(t)})_{t\geq 0}$ for the subordination, we specify their Laplace transform in terms of the memory kernel $k$. To this end we define the function $\Phi:[0,\infty)\times \mathbb{C}\rightarrow \mathbb{C}$ via
\begin{align}
&\Phi(t,\lambda):=\sum_{n=0}^\infty c_n(t)\lambda^n, \label{eq:Phi}\\
&
c_0(t):=1\qquad\forall\,\,t\geq0\nonumber\qquad\text{ and}\\
&
 c_n(t):=\left\{\begin{array}{ll}
\int_0^t k(t,s)c_{n-1}(s)ds, & \quad\forall\,\,t>0,\\
0 & \quad t=0,
\end{array}\right.
\qquad n\in\Nat,\label{eq:coeff of Phi}
\end{align}
It has been shown in~\cite{bebu1} that, under Assumption~\ref{ass:k}, the function $\Phi$ is well-defined (i.e., the integrals in the recursion formula exist) and, for fixed $t$, entire in $\lambda$.

\begin{assumption}\label{ass:Phi}
 Let the function $\Phi$ be constructed from the kernel $k$ via formulas~\eqref{eq:Phi},~\eqref{eq:coeff of Phi}.
   We assume that the restriction of the function $\Phi(t,-\cdot)$ on $(0,\infty)$ is completely monotone 
 for all $t\geq0$, i.e., for each $t\geq0$, there exists  a nonnegative random variable  $A(t)$  whose distribution $\mathcal{P}_{A(t)}$ has the Laplace transform given by $\Phi(t,-\cdot)$: 
 \begin{align}\label{eq:A(t)}
 \int_0^\infty e^{-\lambda a}\mathcal{P}_{A(t)}(da)=\Phi(t,-\lambda),\qquad\forall\,\lambda\in\cC,\quad\Re\lambda\geq0.
 \end{align}
Note that  $\mathcal{P}_{A(0)}=\delta_0$ and $A(0)=0$ a.s. since $\Phi(0,-\lambda)\equiv 1$.
\end{assumption}
Typical examples of kernels $k$ satisfying Assumptions~\ref{ass:k}--\ref{ass:Phi} are kernels of convolution type and homogeneous kernels related to operators of generalized fractional calculus (see  Section~\ref{sec:Examples} below, cf.~\cite{bebu1}). Recall that a  kernel $k$ is  {\it homogeneous of degree} $\theta-1$ for some $\theta>0$ if
\begin{align}\label{eq:homogen k}
k(t,ts)=t^{\theta-1}k(1,s),\qquad\qquad\forall\,\,t\in(0,\infty),\quad s\in(0,1).
\end{align}


\begin{theorem}\label{thm:1}
Let Assumption~\ref{ass:L} hold. Let $k$ satisfy Assumption~\ref{ass:k} and assume that the corresponding function $\Phi$ satisfies Assumption~\ref{ass:Phi}.  Then:

\noindent (i) For each $t\geq0$, the operator $\Phi(t,L)$ given by the Bochner integral
\begin{align}\label{eq:Phi(t,L)}
\Phi(t,L)\ffi:=\int_0^\infty T_a\ffi\,\mathcal{P}_{A(t)}(da),\qquad \ffi\in X,
\end{align}
is well defined and it is a bounded linear operator on $X$.

\noindent (ii) For each $t>0$ and each $u_0\in\Dom(L)$, the function 
\begin{align}\label{eq:u(t)}
u(t):=\Phi(t,L)u_0
\end{align}
solves equation~\eqref{eq:EvEq-general}  and it holds $\lim_{t\searrow 0} u(t)=u_0$.

\noindent (iii)  Suppose additionally that $k$ is homogeneous of order $\theta-1$ for some $\theta>0$. 
 Then one can choose $A(t):=At^\theta$ in~\eqref{eq:Phi(t,L)}, where $A$ is a nonnegative random variable  such that
\begin{equation}\label{eq:A}
\int_0^\infty e^{-\lambda a} \mathcal{P}_{A}(da) = \Phi(1,-\lambda) \qquad \forall \lambda\in\mathbb{C},\quad\Re\lambda\geq0.
\end{equation} 
\end{theorem}

\begin{remark}
Theorem~\ref{thm:1} provides a subordination principle for evolution equations of the form~\eqref{eq:EvEq-general}: Solution~\eqref{eq:u(t)} of equation~\eqref{eq:EvEq-general} is obtained from the solution $T_tu_0$ of the corresponding standard evolution equation $\frac{\pd u}{\pd t}=Lu$, $u(0)=u_0$, via  a ``subordination'' with respect to the ``subordinator'' $(A(t))_{t\geq0}$.
\end{remark}

If the semigroup $(T_t)_{t\geq0}$ has a stochastic representation, then, due to subordination formula~\eqref{eq:Phi(t,L)},  the family $(\Phi(t,L))_{t\geq0}$  has a stochastic representation as well. We present some examples of such stochastic representations below.

\begin{corollary}\label{cor:FKFforMarkovProcesses+V}
(i) Let $Q$ be a Polish\footnote{A Polish space is a separable completely metrizable topological space.} space endowed with a Borel $\sigma$-field $\BB(Q)$ and $(\Om,\FF,\PP^x,(\xi_t)_{t\geq0})_{x\in Q}$ be a (universal) Markov process with state space $(Q,\BB(Q))$. Assume that the corresponding transition semigroup $(T^0_t)_{t\geq0}$, $T^0_t u_0(x):=\E^x\left[ u_0(\xi_t) \right]$, is a strongly continuous semigroup on some Banach space $X\subset B_b(Q)$ (where $B_b(Q)$ is the space of all bounded Borel measurable functions on $Q$). 
Let $(L_0,\Dom(L_0))$ be the generator of $(T^0_t)_{t\geq0}$. Let Assumption~\ref{ass:k} and Assumption~\ref{ass:Phi} hold. Let further $(A(t))_{t\geq0}$ be taken to be independent from $(\xi_t)_{t\geq0}$.  Then, for each $u_0\in\Dom(L_0)$, the function
\begin{align}\label{eq:FKF0}
u(t,x):=\E^x\left[u_0(\xi_{A(t)})  \right], \qquad t\geq0,\quad x\in Q,
\end{align}
solves the evolution equation
\begin{align}\label{eq:FKF2}
u(t,x)&=u_0(x)+\int_0^t k(t,s)L_0 u(s,x)ds, \qquad t>0, \quad x\in Q,\\
\lim_{t \searrow 0} u(t,x) &= u_0(x), \qquad x\in Q. \notag
\end{align}

\noindent (ii) Let additionally $V\,:\,Q\to\cR$ be a Borel measurable function with $\sup_{x\in Q}V(x)\leq c$ for some $c\in\cR$ such that the (closure of the) operator $(L_0+V,\Dom(L_0+V))$ generates a strongly continuous semigroup $(T_t)_{t\geq0}$ on $X$ with stochastic representation\footnote{This is the classical Feynman-Kac formula which holds under very mild assumptions on processes and potentials, cf. \cite{MR1329992} Chapter~3.3.2, \cite{MR1772266,MR3389585}. For example, this Feynman-Kac formula holds in the case $X:=C_\infty(\cRd) =$ the space of continuous functions from $\cRd$ to $\cR$ vanishing at infinity, $(\xi_t)_{t\geq0}$ is a Feller diffusion on $\cRd$, $V\,:\,\cRd\to\cR$ is a bounded continuous function.}
$$
T_tu_0(x):=\E^x\left[ u_0(\xi_t)\exp\left( \int_0^tV(\xi_s)ds \right) \right],\qquad t\geq0,\quad x\in Q,\quad u_0\in X. 
$$
    Then, for each initial condition $u_0\in\Dom(L_0+V)$, the following Feynman-Kac formula
\begin{align}\label{eq:FKF1}
u(t,x):=\E^x\left[ u_0(\xi_{A(t)})\exp\left( \int_0^{A(t)}V(\xi_s)ds \right) \right],\qquad t\geq0,\quad x\in Q,
\end{align}
 provides a solution to the evolution equation
 \begin{align}\label{eq:FKF3}
u(t,x)&=u_0(x)+\int_0^t k(t,s)\big(L_0 u(s,x)+V(x)u(s,x)  \big)ds,\qquad t>0, \quad x\in Q,\\
\lim_{t \searrow 0} u(t,x) &= u_0(x), \qquad x\in Q. \notag
\end{align}

\noindent (iii)  Suppose additionally that $k$ is homogeneous of order $\theta-1$ for some $\theta>0$,
 then one can choose $A(t):=At^\theta$ in~\eqref{eq:FKF0} and in~\eqref{eq:FKF1}, where $A$ is a nonnegative random variable which is independent from $(\xi_t)_{t\geq0}$ and satisfies~\eqref{eq:A}. 
\end{corollary}

\begin{remark}
 If we choose  $(\xi_t)_{t\geq0}$ to be an $\cRd$-valued L\'{e}vy process, $\xi_t:=x+Y_t$ under $\PP^x$, and $X:=C_\infty(\cRd)$, then
Corollary~\ref{cor:FKFforMarkovProcesses+V}~(i) implies 
Theorem 1~(ii) in \cite{bebu1}, where the initial condition $u_0$ is even required to be a member of the Schwartz space of rapdily decreasing smooth functions.
If we take now a bounded continuous potential $V\,:\,\cRd\to\cR$, all assumptions of Corollary~\ref{cor:FKFforMarkovProcesses+V}~(ii) are fullfilled\footnote{In this case, $V$ is a bounded perturbation of $L$; the semigroup generated by $L+V$ exists, is again strongly continuous and has the required Feynman-Kac representation.} and hence the Feynman-Kac formula~\eqref{eq:FKF1} holds. 
Note that, if $V\equiv 0$, 
only one-dimensional marginal distributions of the process $(\xi_{A(t)})_{t\geq0}$ are relevant for the Feynman-Kac formula~\eqref{eq:FKF1} and  the process $(\xi_{A(t)})_{t\geq0}$ can be replaced by any other process with the same one-dimensional marginal distributions. 
If $V$ is a nonzero constant, some particular changings of the process  $(\xi_{A(t)})_{t\geq0}$ are possible. For example, if $(\xi_t)_{t\geq0}$ is a $\delta$-stable L\'{e}vy process,  one may replace $\xi_{A(t)}$ by $\left(A(t)\right)^{1/\delta}\zeta$, where a random variable $\zeta$ has the same distribution as $\xi_1$ and is independent from $(A(t))_{t\geq0}$. In the case of nonconstant potential $V$ it is not possible to change the structure of the process $(\xi_{A(t)})_{t\geq0}$ since 
the whole process $(\xi_s)_{s\geq0}$ is needed 
 in~\eqref{eq:FKF1}. 

\end{remark}

We next wish to apply the semigroup $(T_t)_{t\geq 0}$ associated to an infinitesimal generator $L$ in order to represent the solution of the evolution equation with memory kernel $k$ and the (space-)fractional operator $-(-L)^\gamma$. In order to cover this and related situations, we use subordination in the sense of Bochner~\cite{MR30151,MR50797}.  Recall that subordination in the sense of Bochner is a random time change of a given process $(\xi_t)_{t\geq0}$ by an independent $1$-dimensional increasing L\'{e}vy process (subordinator) $(\eta^f_t)_{t\geq0}$. Any subordinator can be characterized in terms of its Laplace exponent $f$: $\E\left[ e^{-\lambda \eta^f_t} \right]=e^{-tf(\lambda)}$; any such $f$ is a Bernstein function and is determined uniquely by its L\'{e}vy-Khintchine representation
\begin{align*}
f(\lambda)=a+b\lambda+\int_{(0,\infty)}\left(1-e^{-\lambda s}  \right)\nu (ds),
\end{align*} 
where $a$, $b\geq0$ and $\nu$ is a measure on $(0,\infty)$ satisfying $\int_{(0,\infty)}\min(s,1)\nu(ds)<\infty$. Let $(T_t)_{t\geq0}$  be a strongly continuous contraction semigroup on some Banach space $(X,\|\cdot\|_X)$ with generator $(L,\Dom(L))$. The family of operators $(T^f_t)_{t\geq0}$ defined by the Bochner integral
\begin{align*}
T^f_t\ffi:=\int_0^\infty T_s\ffi\,\mathcal{P}_{\eta^f_t}(ds),\qquad\ffi\in X,
\end{align*}
is said to be subordinate to $(T_t)_{t\geq0}$ with respect to the convolution semigroup of measures $\left( \mathcal{P}_{\eta^f_t} \right)_{t\geq0}$, where $\mathcal{P}_{\eta^f_t}$ is the distribution of $\eta^f_t$. The family $(T^f_t)_{t\geq0}$ is again a strongly continuous contraction semigroup on the space  $X$ whose generator $(L^f,\Dom(L^f))$ is the closure of the operator $(-f(-L),\Dom(L))$, where
\begin{align*}
-f(-L)\ffi:=-a\ffi+bL\ffi+\int_{(0,\infty)}\left( T_s\ffi-\ffi \right)\nu(ds),\qquad\ffi\in\Dom(L).
\end{align*}
If $(T_t)_{t\geq0}$ is the transition semigroup of a Feller process $(\xi_t)_{t\geq0}$  and $(\eta^f_t)_{t\geq0}$ is an independent subordinator,  then $(T^f_t)_{t\geq0}$ is the transition semigroup of the (again Feller) process $(\xi_{\eta^f_t})_{t\geq0}$. Further information on subordination in the sense of Bochner and all related objects can be found e.g. in~\cite{MR2978140}.

Consider now the function $\Phi^f(t,-\cdot):=\Phi(t,-f(\cdot))$. If the function $\Phi(t,-\cdot) $ is completely monotone, so is\footnote{As a composition of a Bernstein function $f$ and a completely monotone function $\Phi(t,-\cdot)$.} the function $\Phi^f(t,-\cdot)$. Hence there exists a family of  nonnegative random variables whose Laplace transform is given by $\Phi^f(t,-\cdot)$, $t\geq0$. Using distributions of these random variables and a strongly continuous contraction semigroup $(T_t)_{t\geq0}$ with generator $(L,\Dom(L))$, one can define the operator $\Phi^f(t,L)$ analogously to~\eqref{eq:Phi(t,L)}.

\begin{corollary}\label{cor:subordinate case}
Let Assumption~\ref{ass:L} hold. 
 Let $k$ satisfy Assumption~\ref{ass:k} and the corresponding function $\Phi$ satisfy Assumption~\ref{ass:Phi}. Let $(A(t))_{t\geq0}$ be a family of nonnegative random variables satisfying~\eqref{eq:A(t)}. Let $(\eta^f_t)_{t\geq0}$ be a subordinator corresponding to a Bernstein function $f$ which is independent from $(A(t))_{t\geq0}$.

\medskip

\noindent (i) It holds:
\begin{align*}
\Phi^f(t,L)\ffi=\int_0^\infty T_s\ffi\,\mathcal{P}_{\eta^f_{A(t)}}(ds)=\Phi(t,L^f)\ffi,\qquad\ffi\in X.
\end{align*}
Moreover, for each $t>0$ and each  $u_0\in\Dom(L^f)$, the function  $u(t):=\Phi^f(t,L)u_0$ solves the evolution equation
\begin{align}\label{eq:EvEq-subordinate}
u(t)&=u_0+\int_0^t k(t,s)L^f u(s)ds,\qquad t>0\\
\lim_{t \searrow 0} u(t) &= u_0 \notag. 
\end{align}

\medskip 

\noindent (ii) Let all assumptions of Corollary~\ref{cor:FKFforMarkovProcesses+V} be fulfilled and $L$ in part (i) above be given by  $L:=L_0+V$, where $L_0$ and $V$ are as in Corollary~\ref{cor:FKFforMarkovProcesses+V}. Let additionally $V\leq0$. Let $(\xi_t)_{t\geq0}$ be a Markov process with generator $L_0$ which is independent from $(\eta^f_t)_{t\geq0}$ and $(A(t))_{t\geq0}$. Then  for $u_0\in \Dom((L_0+V)^f)$ the function
\begin{align}\label{eq:FKF+subordinate}
u(t,x):=\E^x\left[u_0\left(\xi_{\eta^f_{A(t)}}  \right) e^{\int_0^{\eta^f_{A(t)}} V(\xi_s)ds} \right]
\end{align}
solves the evolution equation
\begin{align}\label{eq:EvEq+Subordinate+V}
u(t,x)=u_0(x)+\int_0^t k(t,s)\left(L_0+V\right)^fu(s,x)ds.
\end{align}

\medskip

\noindent (iii)  Suppose additionally that $k$ is homogeneous of order $\theta-1$ for some $\theta>0$. Let $A$ be  a nonnegative random variable which satisfies~\eqref{eq:A}  and  is independent from $(\eta^f_t)_{t\geq0}$ and $(\xi_t)_{t\geq0}$. Then we can take $A(t):=At^\theta$ in~\eqref{eq:FKF+subordinate}. 

\end{corollary}

\begin{remark}
 When $L$ is a pseudo-differential operator associated to a L\'evy process and $V\equiv 0$, then we obtain Theorem 1 (iii) in \cite{bebu1} as a special case.
\end{remark}

\begin{remark}
Theorem~\ref{thm:1} can be applied also to generalized time-fractional Schr\"odinger type equations. Note, that different types of fractional analogues of the standard  Schr\"{o}dinger equation  have been discussed in the literature, see, e.g.,~\cite{MR3059867,MR3417088,MR3609236,MR3821542,MR3084964}. Such equations seem to be physically relevant; in particular, some of them arise from the standard  quantum dynamics under special geometric constraints~\cite{MR3965362,MR3919012}. So, let $X:=L^2(\cRd)$ be the Hilbert space of complex-valued square integrable functions; $X$ plays the role of the state space of a quantum system. Let $(\Ham,\Dom(\Ham))$ be a (bounded from below)  self-adjoint operator in $X$  playing the role of the Hamiltonian (energy operator) of this quantum system.  Then $(L,\Dom(L)):=(-i\Ham,\Dom(\Ham))$ does generate a strongly continuous contraction semigroup $(T^\Ham_t)_{t\geq0}$ on $X$ by the Stone theorem. Let $k$, $\Phi$, $(A(t))_{t\geq0}$ be as in Theorem~\ref{thm:1}. Then, by Theorem~\ref{thm:1},
\begin{align}\label{eq:Schr solution}
u(t,x):=\E\left[ T^\Ham_{A(t)}u_0(x) \right]
\end{align}
solves\footnote{Similar results can be found in~\cite{kolokoltsov2017chronological} for the equation~\eqref{eq:Schr eq} with a particular  memory kernel $k$. }  the generalized time-fractional Schr\"odinger type equation
\begin{align}\label{eq:Schr eq}
u(t,x)=u_0(x)-i\int_0^t k(t,s)\Ham u(s,x)ds,
\end{align}
where the equality above is understood as the equality of two elements of the space $X$. 
 For a few particular choices of the Hamiltonian, some stochastic representations of the corresponding semigroup $(T^\Ham_t)_{t\geq0}$ are known in the literature (see, e.g.,~\cite{MR1795612,MR2754894,MR574173}). Inserting these stochastic representations into~\eqref{eq:Schr solution}, one obtains Feynman-Kac formulae (which may be local in the space variables) for the corresponding  generalized time-fractional Schr\"odinger type equation~\eqref{eq:Schr eq}. 
\end{remark}

\begin{remark}
 Under assumptions of Theorem~\ref{thm:1} let $u(t)$ be the solution of equation~\eqref{eq:EvEq-general} given by formula~\eqref{eq:u(t)}. 
 Consider now the following class of time-stretchings $\mathcal{G}:=\big\{g\,:\,[0,\infty)\to[0,\infty)$ such that $g(\tau)\nearrow\infty$ as $\tau\nearrow\infty$, $g(\tau)=\int_0^\tau \dot{g}(\sigma)d\sigma$ for some $\dot{g}\in L^1_{loc}([0,\infty))$, $g(\tau)>0$ and $\dot{g}(\tau)>0$ for all $ \tau>0  \big\}$. The change of variables $t=g(\tau)$, $g\in\mathcal{G}$, shows that  
$$
v(\tau):=u(g(\tau))=\Phi(g(\tau),L)u_0=\int_0^\infty T_a u_0\,\mathcal{P}_{A(g(\tau))}(da)
$$
solves the time-stretched equation
\begin{align}\label{eq:decoratedEvEq}
v(\tau)=u_0+\int_0^\tau\kappa_g(\tau,\sigma) Lv(\sigma)d\sigma,\qquad\tau>0,
\end{align}
where the kernel $\kappa_g$ is defined via
\begin{align*}
\kappa_g(\tau,\sigma):=k(g(\tau),g(\sigma))\dot{g}(\sigma).
\end{align*}
In particular, any stochastic representation of a solution $u(t)$ of equation~\eqref{eq:EvEq-general} induces the corresponding stochastic representation for a solution $v(t)$ of the  time-stretched equation\eqref{eq:decoratedEvEq} for the whole class  $\mathcal{G}$ of time-stretchings.
\end{remark}

\section{Feynman-Kac formulae based on randomely scaled Gaussian processes and further remarks}\label{sec:Examples}

Due to Corollary~\ref{cor:FKFforMarkovProcesses+V} and Corollary~\ref{cor:subordinate case}, the most natural stochastic reprensentations for evolution equations of the form~\eqref{eq:EvEq-general} with $L$ being (a Bernstein function of) the generator of a Markov process (plus a potential term)  are given in terms of  time-changed Markov processes. 
In the special case when the memory kernel $k$ is homogeneous,  one may sometimes use  randomly scaled Gaussian processes (or other randomly scaled processes which are self-similar and have stationary increments) in the obtained stochastic representations. Let us discuss this case. For this recall first a class of memory kernels $k$ from~\cite{bebu1} which satisfy Assumptions~\ref{ass:k}-\ref{ass:Phi} and are homogeneous.

\begin{example}\label{example: k}
 Let  $b>0$, $a\geq b$, $\mu\geq \frac{b}{a}-1$, $\nu>\max\left\{a-b,-a\mu   \right\}$. Consider the Marichev-Saigo-Maeda kernel (cf. Sec.~4 in~\cite{bebu1})
\begin{align}\label{eq:k-Saigo-Maeda}
k(t,s):=\frac{a}{\Gamma(b/a)}(t^a-s^a)^{\frac{b}{a}-1}t^{a-\nu}s^{\nu-1} F_3\left(\frac{\nu}{a}-1,\frac{b}{a},1,\mu, \frac{b}{a}, 1-\left(\frac{s}{t}\right)^a,1-\left(\frac{t}{s}\right)^a\right),
\end{align}
where $0<s<t$ and  $F_3$ is Appell's third generalization   of the Gauss hypergeometric function: $\alpha,\alpha',\beta,\beta',\gamma\in\cC$, $\gamma \notin -\mathbb{N}$,
\begin{align*}
F_3\left(\alpha,\alpha',\beta,\beta',\gamma,x,y\right):=\sum_{m,n\geq 0} \frac{(\alpha)_m(\beta)_m(\alpha')_n(\beta')_n}{(\gamma)_{m+n} n!m!} x^my^n, 
\end{align*}
\begin{align*}
(\delta)_\nu:=\left\{\begin{array}{lll}
1, & \nu=0, & \delta\in\cC\\
\delta(\delta-1)\cdot\ldots\cdot(\delta+n-1), & \nu=n\in\Nat, & \delta\in\cC.
\end{array}\right.
\end{align*}
The kernel $k$  is homogeneous of degree $b-1$ and satisfies Assumption~\ref{ass:k} (cf. Theorem~4 in~\cite{bebu1}). 
 The corresponding function $ \Phi$   has the following form:
\begin{align}\label{eq:Phi Saigo-Maeda}
  \Phi(t,\lambda)= \Gamma(q_2) E_{q_1,q_2}^{q_3}(\lambda t^b),
 \end{align}
 where
\begin{align}\label{eq:lambda 1-2-3}
q_1:=\frac{b}{a},\quad q_2:=\frac{\nu}{a}+\mu,\quad q_3:= 1+\frac{\nu-a}{b},
\end{align}
and  $E_{q_1,q_2}^{q_3}$ is the three parameter Mittag-Leffler (or Prabhakar) function\footnote{The function $E_{q_1,q_2}^{q_3}$ is well-defined on the whole $\cC$ for $\Re q_1>0$ and is an entire function. }
$
E_{q_1,q_2}^{q_3}(\lambda):=\sum_{n=0}^\infty \frac{\left(q_3\right)_n}{\Gamma\left(q_1 n+q_2\right)n!}\,\lambda^n.
$
Under our assumptions on the parameters, the function $\Phi(t,-\cdot)$ is completely monotone and hence Assumption~\ref{ass:Phi} is fulfilled. As corresponding random variables $(A(t))_{t\geq0}$ one may take $A(t):=A_{b,a,\mu,\nu}t^b$, where $A_{b,a,\mu,\nu}$ is a non-negative random variable with Laplace transform $\Gamma(q_2) E_{q_1,q_2}^{q_3}(-\lambda)$. 
 
  Let now $\alpha\in(0,2)$, $\beta\in(0,1]$. In the special case   $b:=\alpha$,  $a:=\frac{\alpha}{\beta}$, $\nu:=a$, $\mu:=0$, the Marichev-Saigo-Maeda kernel~\eqref{eq:k-Saigo-Maeda} reduces to the kernel which appears in the governing equation for generalized grey Brownian motion:
\begin{align}\label{eq:GGBM k}
k(t,s):=\frac{\alpha}{\beta\Gamma(\beta)}s^{\frac{\alpha}{\beta}-1}\left(t^{\frac{\alpha}{\beta}}-s^{\frac{\alpha}{\beta}}  \right)^{\beta-1},\qquad\beta\in(0,1],\quad\alpha\in(0,2).
\end{align}
The corresponding function $\Phi$ in~\eqref{eq:Phi Saigo-Maeda} reduces to the classical Mittag-Leffler function: 
$\Phi(t,\lambda)=E_\beta(\lambda t^\alpha)$.
 And, as the corresponding random variables $(A(t))_{t\geq0}$, one may take $A(t):=A_{\beta}t^\alpha$, where $A_{\beta}$ is a non-negative random variable with Laplace transform $E_\beta(-\cdot)$. For $\beta\in(0,1)$, such random variable $A_\beta$
has 
probability density function 
 $M_\beta\1_{[0,\infty)}$, where  $M_\beta(z):=\sum_{n=0}^\infty \frac{(-z)^n}{n! \Gamma(-\beta n+(1-\beta))}$ is  the Mainardi-Wright function. Generally, probability density function of $A_{b,a,\mu,\nu}$ (when exists) is given in terms of Fox $H$-functions (cf. Remark~10 of~\cite{bebu1}).
\end{example}

Let us now present some  Feynman-Kac formulae for evolution equations of type~\eqref{eq:EvEq-general} with homogeneous kernel $k$ on the base of randomely scaled Gaussian processes.

\begin{example}\label{example:with drift}
(i) Under the assumptions of Corollary~\ref{cor:subordinate case} consider the Bernstein function $f(\lambda):=\lambda^\gamma$, $\gamma\in(0,1]$. Then the operator $L^f$ is the fractional power of the operator $L$, i.e. $L^f=-(-L)^\gamma$ (cf.~\cite{MR0399953,MR2978140}), and $(\eta^f_t)_{t\geq0}$ is a $\gamma$-stable subordinator. Let $k$ be homogeneous of degree $\theta-1$ for some $\theta>0$ and take $A(t)=At^\theta$ according to Corollary~\ref{cor:subordinate case}~(iii). Then the random variable $\eta^f_{A(t)}$ has the same distribution as $A^{1/\gamma}\eta^f_1 t^{\theta/\gamma}$. We may replace the ``subordinator''  $\left(\eta^f_{A(t)}\right)_{t\geq0}$ in~\eqref{eq:FKF+subordinate} by a new ``subordinator'' $\left(\aA t^{\theta/\gamma}\right)_{t\geq0}$ with 
\begin{align}\label{eq:new RV}
\aA:=A^{1/\gamma}\eta^f_1.
\end{align}
 This allows to split randomness and time-dependence in the random time-change. Thus, we obtain the following Feynman-Kac formula 
\begin{align}\label{eq:goodFKF}
u(t,x):&=\E^x\left[u_0\left(\xi_{\aA t^{\theta/\gamma}}  \right) \exp\left(\int_0^{\aA t^{\theta/\gamma}} V(\xi_s)ds\right) \right]\\
&
=\E^x\left[u_0\left(\xi_{\aA t^{\theta/\gamma}}  \right) \exp\left({\aA\frac{\theta}{\gamma}\int_0^{t} s^{\frac{\theta}{\gamma}-1} V(\xi_{\aA s^{\theta/\gamma}})ds}\right) \right]\nonumber
\end{align}
for the evolution equation
\begin{align*}
u(t,x)=u_0(x)-\int_0^t k(t,s) \left(-L_0-V\right)^\gamma u(s,x)ds.
\end{align*}

\noindent (ii)  Let $k$, $A$, $(\eta^f_t)_{t\geq0}$ and  $\aA$ be as in part (i) of this example.  Let $V:=c$ for some $c\leq0$, $\xi_t:=x+B_t+wt$ under $\PP^x$, where $(B_t)_{t\geq0}$ is a standard $d$-dimensional Brownian motion, which is independent from $A$ and $(\eta^f_t)_{t\geq0}$, and $w\in\cRd$ is some fixed vector. Let $X^{\aA,\gamma,\theta}_t:=B_{\aA t^{\theta/\gamma}}$ or $X^{\aA,\gamma,\theta}_t:=\sqrt{\aA} B_{ t^{\theta/\gamma}}$, or, if $H:=\frac{\theta}{2\gamma}\in(0,1)$,  $X^{\aA,\gamma,\theta}_t:=\sqrt{\aA} B^H_{ t}$, where $\left(B^{H}_t  \right)_{t\geq0}$ is  a $d$-dimensional fractional Brownian motion\footnote{Recall that a 1-dimensional fractional Brownian motion $(B^H_t)_{t\geq 0}$ with Hurst parameter $H\in(0,1)$ is a centered Gaussian process with covariance structure ${\E}[B^H_tB_s^H]=\frac{1}{2}(t^{2H}+s^{2H}-|t-s|^{2H})$. A $d$-dimensional fractional Brownian motion with Hurst parameter $H$ is a vector of $d$ independent 1-dimensional ones.}  with Hurst parameter $H$ which is independent from $A$ and $(\eta^f_t)_{t\geq0}$.  Note that  all three options of the process $(X^{\aA,\gamma,\theta}_t)_{t\geq0}$ have the same one-dimensional marginal distributions.   Then, due to Feynman-Kac formula~\eqref{eq:goodFKF},  
\begin{align}\label{eq:FFK-BrownianM+subordination}
u(t,x)=\E\left[u_0\left(x+X^{\aA,\gamma,\theta}_t+\aA wt^{\theta/\gamma}  \right)e^{c\aA t^{\theta/\gamma}}  \right],
\end{align}
solves the evolution equation
\begin{align}\label{eq:EvEq-elliptic L + gamma}
u(t,x)=u_0(x)-\int_0^t k(t,s) \left(-\frac12\Delta-w\nabla-c\right)^\gamma u(s,x) ds.
\end{align}
 Therefore, we have obtained  a Feynman-Kac formula~\eqref{eq:FFK-BrownianM+subordination}  for the evolution equation~\eqref{eq:EvEq-elliptic L + gamma} in terms of two different classes of  randomly scaled Gaussian processes: randomly scaled slowed-down / speeded-up Brownian motion $\left( \sqrt{\aA} B_{ t^{\theta/\gamma}} \right)_{t\geq0}$ and (if $H:=\frac{\theta}{2\gamma}\in(0,1)$)    randomly scaled fractional Brownian motion $\left( \sqrt{\aA} B^H_{ t} \right)_{t\geq0}$.
If $k$ is a Marichev-Saigo-Maeda kernel~\eqref{eq:k-Saigo-Maeda}  then $\theta=b$, $A=A_{b,a,\mu,\nu}$ in distribution. 
  In the special case of the GGBM-kernel~\eqref{eq:GGBM k}, we have $\theta=\alpha$, $A=A_\beta$ in distribution, 
   and  hence we may use generalized grey Brownian motion in formula~\eqref{eq:FFK-BrownianM+subordination} as the process $(X^{\aA,\gamma,\theta}_t)_{t\geq0}$.

 \end{example}


The result of Example~\ref{example:with drift}~(ii) can be generalized beyond the case of a constant diffusion coefficient, as detailed in the case of dimension $d=1$ in space in the following proposition. As can be seen from the proof,  this generalization requires to move from a Brownian motion to a stochastic differential driven by a Brownian motion in the Stratonovich sense in order to apply Corollary \ref{cor:subordinate case}.

\begin{proposition}\label{prop:CB}
Let $\gamma\in(0,1]$ and suppose the kernel $k$ is homogeneous of order $\theta-1$ for some $\theta>0$ and  Assumption~\ref{ass:k}, Assumption~\ref{ass:Phi} are satisfied. Let  $\aA$ be a non-negative random variable constructed by~\eqref{eq:new RV} in Example~\ref{example:with drift}~(i).  Assume  $w\in\cR$,  $c\leq0$, and $\sigma\in C^2(\cR)$ is a bounded function 
  with bounded first and second derivatives.  Consider the linear operator $(L_{(\sigma,w)},\Dom(L_{(\sigma,w)}))$  in $C_\infty(\cR)$ which is defined by 
\begin{align*}
&L_{(\sigma,w)}\ffi(x):=\frac{\sigma^2(x)}{2}\frac{d^2}{d x^2}\ffi(x)+\left(w+\frac12\sigma'(x)  \right)\sigma(x)\frac{d }{d x}\ffi(x),\qquad \ffi\in \Dom(L_{(\sigma,w)}),\\
&
\Dom(L_{(\sigma,w)}):=\left\{ \ffi\in C^2(\cR)\,\,:\,\, \ffi,\, L_{(\sigma,w)}\ffi\in  C_\infty(\cR)\right\}.
\end{align*}
Let $u_0\in \Dom(L_{(\sigma,w)})$ and denote by $g_\sigma\,:\,[0,\infty)\times\cR\to\cR$  the solution to the parametrized family of ODEs
\begin{align}\label{eq:parametrizedODEs}
\frac{\pd}{\pd y}g_\sigma(y,x)=\sigma(g_\sigma(y,x)),\qquad g_\sigma(0,x)=x.
\end{align}
Let $(B_t)_{t\geq0}$ be a standard Brownian motion independent from $\aA$. 
 Let $X^{\aA,\gamma,\theta}_t:=B_{\aA t^{\theta/\gamma}}$ or $X^{\aA,\gamma,\theta}_t:=\sqrt{\aA} B_{ t^{\theta/\gamma}}$, or, if $H:=\frac{\theta}{2\gamma}\in(0,1)$,  $X^{\aA,\gamma,\theta}_t:=\sqrt{\aA} B^H_{ t}$, where $\left(B^{H}_t  \right)_{t\geq0}$ is  a $1$-dimensional fractional Brownian motion  with Hurst parameter $H$ which is independent from $\aA$. 
Then 
\begin{align}
u(t,x)&=\E\left[ u_0\left(g_\sigma\left(X^{\aA,\gamma,\theta}_t+w\aA t^{\theta/\gamma},x \right)  \right) e^{c\aA t^{\theta/\gamma}} \right]\label{eq:FKF-CB+drift}\\
&
=\E\left[ u_0\left(g_\sigma\left(X^{\aA,\gamma,\theta}_t,x \right)  \right) e^{\aA t^{\theta/\gamma}\left(c-\frac{w^2}{2}  \right)+wX^{\aA,\gamma,\theta}_t} \right]\label{eq:FKF-CB-drift}.
\end{align}
solves the evolution equation
\begin{align}\label{eq:EvEq-CB}
u(t,x)=u_0(x)-\int_0^t k(t,s)\left(- L_{(\sigma,w)}-c \right)^\gamma u(s,x)ds.
\end{align}
\end{proposition}
The proof of Proposition~\ref{prop:CB} will be given in Section~\ref{sec:proofs}.
\begin{remark}
Let  $H:=\frac{\theta}{2\gamma}\in(0,1)$ and $X^{\aA,\gamma,\theta}_t:=\sqrt{\aA} B^H_{ t}$, where $\left(B^{H}_t  \right)_{t\geq0}$ is  a $1$-dimensional fractional Brownian motion  with Hurst parameter $H$ 
as in Proposition~\ref{prop:CB}. 
 We now interpret the Feynman-Kac  formula~\eqref{eq:FKF-CB-drift} from the point of view of stochastic differential equations with respect to $\left( X^{\aA,\gamma,\theta}_t \right)_{t\geq0}$ in the rough path sense. Assume $H>1/3$. Then almost every path of $\left( X^{\aA,\gamma,\theta}_t \right)_{t\geq0}$ is H\"older continuous with some index larger than $1/3$.  Let $\mathbb{X}_{t,s}:=\frac12\left(X^{\aA,\gamma,\theta}_t-X^{\aA,\gamma,\theta}_s  \right)^2$. Then $\mathcal{X}:=(X^{\aA,\gamma,\theta}, \mathbb{X})$ is a lift to a geometric rough path (see~\cite{MR4174393}). Consider  $Z^x_t:=g_\sigma\left(X^{\aA,\gamma,\theta}_t,x \right)$. Then, by the It\^{o} formula for geometric rough paths, see again \cite{MR4174393},
 \begin{align}\label{eq: rough SDE}
 Z^x_t=x+\int_0^t\sigma(Z^x_s)dX^{\aA,\gamma,\theta}_s,
\end{align}  
since $g_\sigma\in C^3(\cR)$. Hence, the stochastic representation for the solution of \eqref{eq:EvEq-CB} in \eqref{eq:FKF-CB-drift} can be rewritten in the form
$$
u(t,x)=\E\left[ u_0\left(Z^x_t\right) e^{\aA t^{\theta/\gamma}\left(c-\frac{w^2}{2}  \right)+wX^{\aA,\gamma,\theta}_t} \right]
$$
This form resembles the classical Feynman-Kac formula for parabolic Cauchy problems in terms of stochastic differential equations driven by a Brownian motion. However, the stochastic differential equation \eqref{eq: rough SDE}  is driven by a randomly scaled fractional Brownian motion, which is neither a semimartingale nor a Markov process (unless $H=1/2$), to account for the memory kernel and the space fractionality in \eqref{eq:EvEq-CB}, while maintaining the stationary increment property of the driving process.
\end{remark}

\bigskip
\newpage
Let us now discuss evolution equations of the form~\eqref{eq:EvEq-general} in the special case, when the  kernel $k$  is of convolution type.

\begin{remark}
 Suppose the kernel $k$ has the form   $k(t,s):=\mathfrak{K}(t-s)$, where $\mathfrak{K}:(0,\infty)\rightarrow \mathbb{R}$ is continuous and satisfies $
 |\mathfrak{K}(t)|\leq M t^{\beta-1}e^{\gamma t},\quad t>0,
 $
 for some constants $M,\gamma\geq 0$ and $\beta \in (0,1]$. Let $(\mathcal{L}\mathfrak{K})(\cdot)$ be the Laplace transform of $\mathfrak{K}$.  By  Theorem~3 in~\cite{bebu1}, if there exists  a nonnegative stochastic process $(A(t))_{t\geq 0}$ such that almost all its paths are   right-continuous with left limits and   such that  
 \begin{align}
 \int_0^\infty e^{-\sigma t} \mathbb{E}\left[e^{-\lambda A(t)}\right]dt= \frac{1}{\sigma}\frac{1}{1+\lambda (\mathcal{L}\mathfrak{K})(\sigma)}
\end{align}
 for every $\lambda\geq 0$ and sufficiently large $\sigma\geq \sigma_0(\lambda)$, then the function $\Phi(t,-\cdot)$ constructed from $k(t,s)=\mathfrak{K}(t-s)$ by~\eqref{eq:Phi},~\eqref{eq:coeff of Phi} is completely monotone for every $t\geq 0$ and the above process $(A(t))_{t\geq0}$ satisfies~\eqref{eq:A(t)}. In particular, consider the case when   $\mathcal{L}\mathfrak{K}=1/h$ for some Bernstein function $h$. Then $h$ is the Laplace exponent of some L\'{e}vy subordinator $(\eta^h_t)_{t\geq0}$. The corresponding inverse subordinator $(\mathcal{E}^h_t)_{t\geq0}$  is defined via $\mathcal{E}^h_t:=\inf\left\{ s>0\,:\,\eta^h_s>t\right\}$.   It has been shown in~\cite{MR2442372} (formula~(3.14)) that (in the case when the L\'{e}vy measure $\nu$ of $(\eta^h_t)_{t\geq0}$ satisfies $\nu(0,\infty)=\infty$) the double Laplace  transform 
of the distribution $\mathcal{P}_{\mathcal{E}^h_t}(da)$ with respect to both time and space variables is equal to
$$
\int_0^\infty e^{-\sigma t} \mathbb{E}\left[e^{-\lambda \mathcal{E}^h_t}\right]dt=\frac{h(\sigma)}{\sigma(h(\sigma)+\lambda)}= \frac{1}{\sigma}\frac{1}{1+\lambda (\mathcal{L}\mathfrak{K})(\sigma)}.
$$
Hence, one may take $A(t):=\mathcal{E}^h_t$, $t\geq0$, in this case. Note that the assumption $\nu(0,\infty)=\infty$ guarantees that the sample paths of $(\eta^h_t)_{t\geq0}$ are a.s. strictly increasing, i.e. almost all paths $t\mapsto \mathcal{E}^h_t$ are continuous (cf. \cite{MR3373947}). 
\end{remark}

\begin{example}
Let us mention the following functions $\mathfrak{K}_1$  and $\mathfrak{K}_2$  providing admissible  kernels $k$ of convolution type and having Laplace transform $1/h$ for some Bernstein function $h$ (cf.~\cite{MR4223052}): for $1\geq\beta>\beta_1>\ldots>\beta_m>0$, $b_j>0$, $j=1,\ldots,m$
\begin{align*}
\mathfrak{K}_1(t):=\frac{t^{\beta-1}}{\Gamma(\beta)}+\sum_{j=1}^mb_j\frac{t^{\beta_j-1}}{\Gamma(\beta_j)} 
\end{align*}
 with the corresponding Bernstein function $h_1(\sigma):=\left(\sigma^{-\beta}+\sum_{j=1}^m b_j\sigma^{-\beta_j}  \right)^{-1}$ and
\begin{align*}
\mathfrak{K}_2(t):=t^{\beta-1}{E}_{(\beta-\beta_1,\ldots,\beta-\beta_m),\beta}\left( -b_1 t^{\beta-\beta_1},\ldots,-b_m t^{\beta-\beta_m} \right)
\end{align*}
with the multinomial Mittag-Leffler function~\cite{MR1467147,MR1366608} (for $z_j\in\cC$, $\beta\in\cR$, $\alpha_j>0$, $j=1,\ldots,m$)
\begin{align*}
{E}_{(\alpha_1,\ldots,\alpha_m),\beta}(z_1,\ldots,z_m):=\sum_{n=0}^\infty\sum_{\begin{array}{c}
n_1+\ldots+n_m=n\\
n_1\in\Nat_0,\ldots,n_m\in\Nat_0
\end{array}}
\frac{n!}{n_1!\cdots n_m!}\frac{\prod_{j=1}^m z_j^{n_j}}{\Gamma\left( \beta+\sum_{j=1}^m \alpha_jn_j  \right)}.
\end{align*} 
The kernel $\mathfrak{K}_2$ corresponds to the Bernstein function $h_2(\sigma):=\sigma^{\beta}+\sum_{j=1}^mb_j\sigma^{\beta_j}$. The corresponding functions $\Phi_1(t,-\lambda)$ and $\Phi_2(t,-\lambda)$  are found in~\cite{MR4223052} in terms of the multinomial Mittag-Leffler  function:
\begin{align*}
&\Phi_1(t,-\lambda):= {E}_{(\beta,\beta_1,\ldots,\beta_m),1}\left(-\lambda t^\beta,-\lambda t^{\beta_1},\ldots,-\lambda t^{\beta_m}  \right),\\
&
\Phi_2(t,-\lambda):=1-\lambda t^{\beta}E_{(\beta,\beta-\beta_1,\ldots,\beta-\beta_m),\beta+1}\left(-\lambda t^\beta,-\lambda t^{\beta_1},\ldots,-\lambda t^{\beta_m}  \right).
\end{align*}
\end{example}

\begin{remark}
Note that, in the case $k(t,s)=\mathfrak{K}(t-s)$ with $\mathcal{L}\mathfrak{K}=1/h$ for some Bernstein function  $h$, evolution equation~\eqref{eq:EvEq-general} is equivalent (what can be shown by applying the Laplace transform w.r.t. time-variable to both equations) to the Cauchy problem  
\begin{align}\label{eq:Caputo form}
\mathcal{D}^h_t u(t,x)=Lu(t,x),\qquad u(0,x)=u_0(x),\qquad x\in\cRd,\quad t>0,
\end{align}
where $\mathcal{D}^h_t$ is a generalized time-fractional derivative of Caputo type, which is defined (for sufficiently good functions $v\,:\,(0,\infty)\to\cR$ of time variable $t$) via the Laplace transform (cf.~\cite{MR4098657}) by
\begin{align*}
\left(\mathcal{L}\left[\mathcal{D}^h_t v\right]\right)(\sigma)=h(\sigma)(\mathcal{L}v)(\sigma)-\frac{h(\sigma)}{\sigma}v(+0).
\end{align*}
Therefore, the results of  Theorem~\ref{thm:1} and Corollaries~\ref{cor:FKFforMarkovProcesses+V},~\ref{cor:subordinate case}    provide  solutions  for evolution equations of the form~\eqref{eq:Caputo form} with  generalized time-fractional derivatives of Caputo type $\mathcal{D}^h_t$. 
 In the case $h(\sigma):=\sigma^\beta$, $\beta\in(0,1)$, the generalized time-fractional derivative $\mathcal{D}^h_t$ coincides with the Caputo derivative of order $\beta$. 
The kernel $\mathfrak{K}_1$ corresponds to  a  mixture of Caputo time-fractional derivatives of orders $\beta$, $\beta_1,\ldots,\beta_m$. In the case of Bernstein function $h(\sigma):=\int_0^1 \sigma^\beta\mu(d\beta)$ with a finite Borel measure $\mu$ concentrated on the interval $(0,1)$, the corresponding derivative $\mathcal{D}^h_t$ is known as \emph{distributed order fractional derivative}. 
\end{remark}


\section{Proofs}\label{sec:proofs}
\begin{proof}[Proof of Theorem~\ref{thm:1}]
(i) Let Assumptions~\ref{ass:L},~\ref{ass:k} and~~\ref{ass:Phi} hold. Since the function $\Phi(t,\cdot)$, $t\geq0$, is entire by Theorem 1 in \cite{bebu1}, the function $\Phi(t,i(\cdot))$ is also entire and is the characteristic function of the distribution $\mathcal{P}_{A(t)}$, which is concentrated on $[0,\infty)$. Therefore, we have by the Raikov theorem (cf. Theorem~3.2.1 in~\cite{MR0125601}) 
\begin{align}\label{eq:estimate}
\int_\cR e^{r|a|}\mathcal{P}_{A(t)}(da)=\int_0^\infty e^{ra}\mathcal{P}_{A(t)}(da)<\infty\qquad\forall\,\, r>0.
\end{align}  
Further, for any strongly continuous semigroup $(T_t)_{t\geq0}$ there exist constants $M\geq1$, $c\geq0$ such that $\|T_t\|\leq M e^{ct}$,  $\forall\,t\geq0$, and the mapping $t\mapsto T_t\ffi$ is continuous for any $\ffi\in X$. Thus,  we have $\int_0^\infty\|T_a\ffi\|_X \mathcal{P}_{A(t)}(da)<\infty$ and the Bochner integral in the r.h.s. of~\eqref{eq:Phi(t,L)} is well defined for any $\ffi\in X$. Moreover, the operator $\Phi(t,L)$ defined by~\eqref{eq:Phi(t,L)} is a bounded linear operator on $X$ and $\Phi(0,L)=\Id$.

\medskip

\noindent (ii) Recall that the following statement was proved in~\cite{bebu1} (cf. Corollary~1 of~\cite{bebu1}):
\begin{lemma}\label{lem:1}
Let Assumption~\ref{ass:k} hold. 
Then, for each $\lambda\in\cC$, there exists a unique solution $\Phi(\cdot,-\lambda)\in B_b([0,T],\cC)$, $\forall\,\, T>0$, 
 of the following Volterra  equation of the second kind
\begin{align}\label{eq:Volterra for Phi}
\Phi(t,-\lambda)=1-\lambda\int_0^t k(t,s)\Phi(s,-\lambda)ds,\qquad t>0.
\end{align}
Moreover, $\lim_{t\searrow0}\Phi(t,-\lambda)=1$ locally uniformly with respect to $\lambda\in\cC$,
 $\Phi(t,\cdot)$ is an entire function for  all $t\geq0$ and equalities~\eqref{eq:Phi} and~\eqref{eq:coeff of Phi} hold. 
\end{lemma}
\noindent Our aim is to lift the equality~\eqref{eq:Volterra for Phi} to the level of operators $\Phi(t,L)$. To this aim we use the so-called Hille-Phillips functional calculus. Let us recall the main facts about this functional calculus (cf.~\cite{MR2244037,ISem21}). 

Let $(T_t)_{t\geq0}$ be as in Assumption~\ref{ass:L}.  Consider first the case when  $(T_t)_{t\geq0}$ is uniformly bounded (i.e. $\|T_t\|\leq M$ for some $M\geq1$ and all $t\geq0$). Denote by  $LS(\cC_+)$ the space of functions that are Laplace transforms of complex measures on  $([0,\infty),\mathscr{B}([0,\infty)))$. Let $g\in LS(\cC_+)$ and $m_g$ be the (unique) complex measure whose Laplace transform $\mathscr{L}[m_g]$ is given by $g$. One defines the operator $g(-L)$ as follows:
\begin{align}\label{eq:g(-L)}
g(-L)\ffi:=\int_0^\infty T_a\ffi\,m_g(da),\qquad\ffi\in X.
\end{align}
The right hand side of~\eqref{eq:g(-L)} is a well-defined Bochner integral and $g(-L)$ is a bounded linear operator on $X$, i.e. $g(-L)\in\mathcal{L}(X)$. The mapping $\calc_T\,:\,LS(\cC_+)\to \mathcal{L}(X)$, $g\mapsto g(-L)$,  is called the \emph{Hille-Phillips calculus} for $-L$. Note that $\calc_T$ is an algebra homomorphism and hence  $\calc_T(g_1g_2)= g_1(-L)\circ g_2(-L)=g_2(-L)\circ g_1(-L)$ and  $ \calc_T( a g_1+b g_2)= a g_1(-L)+b g_2(-L)$ for any $g_1$, $g_2\in LS(\cC_+)$,  $a,b\in\cR$.

 Consider now the case when  $(T_t)_{t\geq0}$ is of type $c\geq0$ (i.e., $\|T_t\|\leq M e^{ct}$ for some $M\geq1$, $c\geq0$ and all $t\geq0$). Then the rescaled semigroup $(T^c_t)_{t\geq0}$, $T^c_t:=T_te^{-ct}$, is uniformly bounded, strongly continuous and has generator $(L-c,\Dom(L))$. Then one may use the Hille-Phillips calculus $\calc_{T^c}$ for $-(L-c)$.  Consider now the space $LS(\cC_+-c):=\left\{g\,:\,g(\cdot-c)\in  LS(\cC_+)  \right\}$.  Let $g\in LS(\cC_+-c)$ and $m^c_g$ be the (unique) complex measure with $\mathscr{L}[m^c_g]=g(\cdot-c)$.  One defines the operator $g(-L)$ as follows:
\begin{align*} 
g(-L)\ffi:=\calc_{T^c}\big(g(\cdot-c)  \big)\ffi\equiv \int_0^\infty T^c_a\ffi\,m^c_g(da),\qquad\ffi\in X.
\end{align*}
Let now $m$ be a complex measure such that $e^{ca}m(da)$ is again a complex measure. Let $g^*:=\mathscr{L}[m]$. Then it holds
\begin{align*}
\mathscr{L}[e^{ca}m(da)](\lambda)=\int_0^\infty e^{-\lambda a}e^{ca}m(da)=g^*(\lambda-c).
\end{align*}
Hence $g^*\in LS(\cC_+-c)$  and $m^c_{g^*}(da)=e^{ca}m(da)$. Therefore, 
$$
g^*(-L)\ffi=\int_0^\infty T^c_a\ffi\,m^c_{g^*}(da)=\int_0^\infty T_a\ffi\, m(da),\qquad\ffi\in X.
$$
Thus, the operator $\Phi(t,L)$ defined in~\eqref{eq:Phi(t,L)} can be interpreted in terms of Hille-Phillips calculus as  $\calc_{T^c}\big(\Phi(t,-(\cdot-c)) \big)$ due to~\eqref{eq:estimate}.

Now we are ready to transfer equality~\eqref{eq:Volterra for Phi} to the level of operators by means of Hille-Phillips calculus.  Let $(T_t)_{t\geq0}$ be of type $c\geq0$ and $\rho(L)$ be the resolvent set of operator $L$, i.e. the resolvent operator $R_\lambda(L):=(\lambda-L)^{-1}$ is a well defined bounded operator on $X$ for each $\lambda\in\rho(L)$. Let $\gamma>c$. Hence $\gamma\in\rho(L)$. And equality~\eqref{eq:Volterra for Phi} implies the equality
\begin{align}\label{eq:(1)}
\gamma \cdot \frac{\Phi(t,-\lambda) -1}{\gamma + \lambda} = -\lambda\cdot\frac{\gamma}{\gamma+\lambda} \cdot \int_0^t k(t,s)\Phi(s,-\lambda)ds,\quad \forall t>0,\ \forall \lambda\in\mathbb{C}: \Re\lambda\geq -c.
\end{align}
Let us present each component of~\eqref{eq:(1)} as the Laplace transform of some complex measure on $([0,\infty),\mathscr{B}([0,\infty)))$.  As we have already discussed
\begin{align*}
\Phi(t,-\lambda) =  \mathscr{L}(\mathcal{P}_{A(t)})(\lambda)\overset{\calc_{T^c}}{  \longleftrightarrow}  \Phi(t,L)= \int_0^\infty T_a \,\mathcal{P}_{A(t)}(da). 
\end{align*}
Furthermore, we have with Dirac delta-measure $\delta_0$ and with exponential distribution $Exp(\gamma)$:
\begin{align*}
1&= \mathscr{L}(\delta_0)(\lambda)  \overset{\calc_{T^c}}{\longleftrightarrow} \mathscr{L}(\delta_0)(-L):= \int_0^\infty T_a^c \delta_0(da) = Id, \\
\frac{\gamma}{\gamma + \lambda} & = \mathscr{L}(Exp(\gamma))(\lambda) \overset{\calc_{T^c}}{\longleftrightarrow} 
\mathscr{L}(Exp(\gamma))(-L):= \int_0^\infty T_a^c  \gamma e^{-\gamma a}e^{c a} da\\
&
\phantom{bgfbbdgbfgrgrggegrvff} = \int_0^\infty T_a \gamma e^{-\gamma a} da = \gamma\cdot (\gamma -L)^{-1}=\gamma\cdot R_\gamma(L).
\end{align*}
Note that $R_\gamma(L)$ is a bounded operator since $\gamma\in (c,\infty)\subset\rho(L)$ (cf.~\cite{MR1721989,MR710486}) and $\|\gamma R_\gamma(L)\ffi-\ffi\|_X\to0$ as $\gamma\to\infty$ for any $\ffi\in X$. Further,
\begin{align*}
\frac{-\lambda\gamma}{\gamma+\lambda} &= -\gamma\cdot 1+ \gamma\cdot\frac{\gamma}{\gamma+\lambda} = -\gamma\cdot \mathscr{L}(\delta_0)(\lambda) + \gamma \cdot \mathscr{L}(Exp(\gamma))(\lambda) \overset{\calc_{T^c}}{\longleftrightarrow} \\
	&-\gamma\cdot\mathscr{L}(\delta_0)(-L) + \gamma \cdot \mathscr{L}(Exp(\gamma))(-L)= -\gamma\cdot Id + \gamma^2 R_\gamma (L) =: L_\gamma.
\end{align*}
Note that  $L_\gamma$ is the so-called \emph{ Yosida-Approximation} of $L$ (cf.~\cite{MR1721989,MR710486});  $L_\gamma$ is a bounded operator and $\|L\ffi-L_\gamma\ffi\|_X\to0$ as $\gamma\to+\infty$ for each $\ffi\in\Dom(L)$.\\
Without loss of generality we now assume $k(t,s)\ge 0$ (else divide into negative and nonnegative part) and define a family of measures on $([0,\infty),\mathcal{B}([0,\infty)))$ via
\begin{equation*}
\nu_t(B):= \int_0^t k(t,s) \mathcal{P}_{A(s)}(B)ds, \qquad B\in\mathcal{B}([0,\infty)).
\end{equation*}
The right hand side in the above formula is well-defined since the mapping $s\mapsto\mathcal{P}_{A(s)}(B)$ is a bounded Borel-measurable function on $[0,\infty)$ for any $B\in\mathcal{B}([0,\infty))$. Indeed, the mapping $s\mapsto\Phi(s,-\lambda)$ is Borel measurable for any $\lambda\in\cC$ due to Assumption~\ref{ass:k} and representation formulas~\eqref{eq:Phi},~\eqref{eq:coeff of Phi}. And for any $s,x\geq0$ holds (cf. Lemma~1.1 and the proof of Prop.~1.2 in \cite{MR2978140}) 
$$
\mathcal{P}_{A(s)}([0,x])=\liml_{\lambda\to\infty}\sum_{k\leq \lambda x}(-1)^k\frac{\pd^k\Phi(s,-\lambda)}{\pd \lambda^k}\frac{\lambda^k}{k!}.
$$
Further, it holds for measurable $g:[0,\infty)\to[0,\infty)$  
\begin{equation}\label{eq:(2)}
\int_0^\infty g(a) \nu_t(da) = \int_0^t k(t,s) \int_0^\infty g(a) \mathcal{P}_{A(s)}(da) ds,
\end{equation}
which can be seen via approximation of $g$ by step functions from below and the use of Beppo Levi's Theorem.
By choosing $g(a):=e^{-\lambda a}$ we see that
\begin{equation*}
\int_0^\infty e^{-\lambda a} \nu_t(da) = \int_0^t k(t,s) \int_0^\infty e^{-\lambda a} \mathcal{P}_{A(s)}(da) ds = \int_0^t k(t,s) \Phi(s,-\lambda) ds =:\Psi(t,-\lambda).
\end{equation*}
Thereby $\Psi(t,-\lambda)$ is the Laplace transform of an appropriate measure and we get the correspondence 
\begin{equation*}
\Psi(t,-\lambda) \overset{\calc_{T^c}}{\longleftrightarrow} \Psi(t,L):=\int_0^\infty T_a^c {\nu}^c_t(da)
\end{equation*}
where ${\nu}^c_t(da):=e^{c a}\nu_t(da)$. Note that ${\nu}^c_t$ is a bounded measure on  $([0,\infty),\mathscr{B}([0,\infty)))$ due to~\eqref{eq:estimate}.
Furthermore, similar to property \eqref{eq:(2)}, it holds for any Bochner-integrable function $g\,:\,[0,\infty)\to X$ 
\begin{align*}
\int_0^\infty g(a) {\nu}^c_t(da) = \int_0^t k(t,s) \int_0^\infty g(a) e^{c a} \mathcal{P}_{A(s)}(da) ds,
\end{align*}
and therefore, for any $\ffi\in X$,
\begin{align*}
\Psi(t,L)\ffi &= \int_0^\infty T_a^c \ffi\,{\nu}^c_t(da) = \int_0^t k(t,s) \int_0^\infty T_a^c\ffi\, e^{c a} \mathcal{P}_{A(s)}(da)ds\\
& = \int_0^t k(t,s) \Phi(s,L)\ffi ds. 
\end{align*}
Thereby, all components of~\eqref{eq:(1)} have been transferred. Taking everything together and using that for $u_0\in\Dom(L)$ holds (cf.~\cite{ISem21}) 
\begin{equation*}
L_\gamma \Psi(t,L)u_0 = \gamma L R_\gamma(L) \Psi(t,L) u_0 = \Psi(t,L)\gamma L R_\gamma(L) u_0 = \Psi(t,L) L_\gamma u_0,
\end{equation*}
we get 
\begin{equation}\label{(2)}
\gamma R_\gamma(L)\left(\Phi(t,L) -Id\right)u_0 = \Psi(t,L) L_\gamma u_0  \qquad \forall\, u_0\in \Dom(L).
\end{equation}
Taking the limit $\gamma\to+\infty$ we obtain  (with $\Phi(s,L) L u_0=L \Phi(s,L) u_0$ for all $u_0\in\Dom(L)$)
\begin{align*}
(\Phi(t,L)-Id)u_0 &= \Psi(t,L) L u_0 = \int_0^t k(t,s) \Phi(s,L) L u_0 ds = \int_0^t k(t,s) L \Phi(s,L) u_0 ds\\
	\Leftrightarrow \ \Phi(t,L)u_0 &= u_0 + \int_0^t k(t,s) L\Phi(s,L)u_0 ds.
\end{align*}
Therefore, the function $u(t):=\Phi(t,L)u_0$ solves evolution equation~\eqref{eq:EvEq-general} for  any $u_0\in\Dom(L)$.

For continuity at zero we evaluate equality \eqref{eq:A(t)} at $\lambda=-c-i\rho$, $\rho\in\cR$, resulting in
\begin{equation*}
\int_0^\infty e^{i\rho a} e^{c a} \mathcal{P}_{A(t)}(da) = \Phi(t,i\rho+c)\qquad \forall (t,\rho)\in[0,\infty)\times\cR.
\end{equation*}
According to Lemma \ref{lem:1}
\begin{equation*}
\lim_{t\searrow 0}\int_0^\infty e^{i\rho a} e^{c a} \mathcal{P}_{A(t)}(da) = \lim_{t\searrow 0} \Phi(t,i\rho+c) = 1 \qquad \forall \rho\in\cR,
\end{equation*}
and by L\'{e}vy's Continuity Theorem it follows that
\begin{equation*}
e^{c a}\mathcal{P}_{A(t)}(da) \xrightarrow{\text{weakly}} \delta_0(da),\qquad t\searrow0.
\end{equation*}
We now write
\begin{align*}
\|u(t)-u_0\|_X=\left\|\int_0^\infty \left(T_au_0 -u_0 \right)\mathcal{P}_{A(t)}(da)\right\|_X &\le \int_0^\infty \|T_au_0-u_0\|_X e^{-c a}e^{c a} \mathcal{P}_{A(t)}(da)\\
	 &= \int_\cR f(a)e^{c a} \mathcal{P}_{A(t)}(da),
\end{align*}
where 
\begin{equation*}
f:\cR\to\cR, \qquad a\mapsto \begin{cases} 0, &a<0 \\ \|T_au_0-u_0\|_X e^{-c a}, &a\ge0\end{cases}
\end{equation*}
is a bounded and continuous function. Now
\begin{equation*}
\lim_{t\searrow 0} \|u(t)-u_0\|_X \leq \lim_{t\searrow 0}  \int_\cR f(a)e^{c a} \mathcal{P}_{A(t)}(da) =f(0)= 0
\end{equation*}
by weak convergence and thus continuity at zero is shown.

\medskip

\noindent (iii) Let $k$ be homogeneous of order $\theta-1$ for some $\theta>0$.
By the recursion formula~\eqref{eq:coeff of Phi}  for all $t>0$, $n\in\Nat$
\begin{equation*}
c_n(t)=t^\theta\int_0^1k(1,s)c_{n-1}(ts)ds=t^{n\theta}\int_0^1k(1,s)c_{n-1}(s)ds=t^{n\theta}c_n(1).
\end{equation*}
And, thus, we have for all  $ t\geq0$, $\lambda\in\mathbb{C}$ (cf. Theorem~2 in \cite{bebu1}):
\begin{equation*}
\Phi(1,-t^\theta\lambda) = \sum_{n=0}^\infty c_n(1)\left(-t^\theta\lambda\right)^n = \sum_{n=0}^\infty t^{-n\theta}c_n(t)(-t^\theta\lambda)^n = \sum_{n=0}^\infty c_n(t)(-\lambda)^n = \Phi(t,-\lambda).
\end{equation*}
Let $A(t):=At^\theta$, where $A$ is a nonnegative random variable satisfying~\eqref{eq:A}.  Then 
\begin{equation*}
\mathscr{L}(\mathcal{P}_{A(t)})(\lambda) = \E\left[e^{-\lambda At^\theta}\right] = \mathscr{L}(\mathcal{P}_A)(\lambda t^\theta) = \Phi(1,-\lambda t^\theta) = \Phi(t,-\lambda).
\end{equation*}
Therefore, $A(t):=A t^\theta$ has the required distribution.
Theorem~\ref{thm:1} is proved.
\end{proof}

\begin{proof}[Proof of Corollary~\ref{cor:FKFforMarkovProcesses+V}]
(i) By construction $(T_t^0)_{t\geq0}$ is a strongly continuous semigroup of type $0$ on $X$. It follows from Theorem~\ref{thm:1} ii) and Fubini's theorem that
\begin{align*}
u(t,x) = \Phi(t,L_0)u_0(x)= \int_0^\infty \E^x\left[u_0(\xi_a)\right]\mathcal{P}_{A(t)}(da) = \E^x\left[u_0(\xi_{A(t)})\right]
\end{align*}
solves the evolution equation~\eqref{eq:FKF2}.
\\
(ii) $(T_t)_{t\geq0}$ is a strongly continuous semigroup of type $\max\{c,0\}$ on $X$. It follows from Theorem~\ref{thm:1} ii) and Fubini's theorem that
\begin{align*}
u(t,x) = \Phi(t,L_0+V)u_0(x)&= \int_0^\infty \E^x\left[u_0(\xi_a)exp\left(\int_0^a V(\xi_s)ds\right)\right]\mathcal{P}_{A(t)}(da) \\
	&= \E^x\left[u_0(\xi_{A(t)})exp\left(\int_0^{A(t)} V(\xi_s)ds\right)\right]
\end{align*}
solves the evolution equation~\eqref{eq:FKF3}.
\\
(iii) Follows immediately from Theorem~\ref{thm:1}~(iii).
\end{proof}

\begin{proof}[Proof of Corollary~\ref{cor:subordinate case}]
(i) $(T_t^f)_{t\geq0}$ is a strongly continuous contraction semigroup  on the Banach space $X$. Therefore, $(T_t^f)_{t\geq0},\ k$ and $\Phi$ fulfill all assumptions of Theorem~\ref{thm:1} and thus 
\begin{equation*}
\Phi(t,L^f)\ffi:=\int_0^\infty T^f_s \ffi \mathcal{P}_{A(t)}(ds),\qquad \ffi\in X,
\end{equation*}
is well-defined. 
Let now   $(A(t))_{t\geq0}$ and $(\eta^f_t)_{t\geq0}$ be as in the statement of Corollary~\ref{cor:subordinate case}.    
Consider the family of random variables $\left(\eta^f_{A(t)}\right)_{t\geq0}$. Then
\begin{align*}
\E\left[ e^{-\lambda \eta^f_{A(t)} } \right]=\int_0^\infty\E\left[ e^{-\lambda \eta^f_a} \right]\mathcal{P}_{A(t)}(da)=\int_0^\infty e^{-af(\lambda)}\mathcal{P}_{A(t)}(da)=\Phi(t,-f(\lambda)).
\end{align*}
Starting with  the strongly continuous contraction semigroup $(T_t)_{t\geq0}$ and the completely monotone function $\Phi^f(t,-\cdot):=\Phi(t,-f(\cdot))$, one may define
\begin{equation*}
\Phi^f(t,L) \ffi:=\int_0^\infty T_s \ffi\, \mathcal{P}_{\eta^f_{A(t)}}(ds),\qquad \ffi\in X. 
\end{equation*} 
Due to Fubini's theorem
\begin{align*}
\Phi(t,L^f)\ffi=\int_0^\infty T^f_s \ffi \,\mathcal{P}_{A(t)}(ds) &= \int_0^\infty \int_0^\infty T_a \ffi\, \mathcal{P}_{\eta^f_s} (da) \mathcal{P}_{A(t)}(ds) \\
	&= \int_0^\infty T_a \ffi\, \mathcal{P}_{\eta^f_{A(t)}}(da) = \Phi^f(t,L)\ffi, \qquad \ffi\in X.
\end{align*}
Therefore,  for $u_0\in \Dom(L^f)$, equation~\eqref{eq:EvEq-subordinate} is solved by $\Phi(t,L^f)u_0=\Phi^f(t,L)u_0$ according to Theorem~\ref{thm:1}~(ii).
\\
(ii) Since $V\le0$, $(T_t)_{t\geq0}$ is a strongly continuous contraction semigroup and so is $(T_t^f)_{t\geq0}$. It follows from Theorem~\ref{thm:1}~(ii) that $u(t,x):=\Phi(t, (L+V)^{f})u_0$ solves evolution equation~\eqref{eq:EvEq+Subordinate+V} and due to Fubini's theorem
\begin{align*}
\Phi(t, (L+V)^{f})u_0 &= \int_0^\infty T_a^{f} u_0 \mathcal{P}_{A(t)}(da) \\
	&= \int_0^\infty \int_0^\infty \E^x\left[u_0(\xi_s)exp\left(\int_0^s V(\xi_v)dv\right)\right] \mathcal{P}_{\eta_a^f}(ds) \mathcal{P}_{A(t)}(da)\\
	&= \int_0^\infty \E^x\left[u_0(\xi_{\eta_a^f})exp\left(\int_0^{\eta_a^f} V(\xi_v)dv\right)\right] \mathcal{P}_{A(t)}(da) \\
	&= \E^x\left[u_0(\xi_{\eta_{A(t)}^f})exp\left(\int_0^{\eta_{A(t)}^f} V(\xi_s)ds\right)\right].
\end{align*}
\\
(iii) Follows immediately from Theorem~\ref{thm:1}~(iii).

\end{proof}

\begin{proof}[Proof of Proposition~\ref{prop:CB}]
First, note that, under our assumptions on $\sigma$, the   operator  $(L_{(\sigma,w)},\Dom(L_{(\sigma,w)}))$ does generate a strongly continuous semigroup on $C_\infty(\cR)$ (cf.~\cite{MR3012216}, Sec.~3.1.2).  Second, consider the pair $\left((\xi_t)_{t\geq0},(\PP^x)_{x\in\cR}\right)$ where $(\xi_t)_{t\geq0}$ solves the Stratonovich SDE with respect to a standard $1$-dimensional Brownian motion $(B_t)_{t\geq0}$
\begin{align*}
d\xi_t=\sigma(\xi_s)\circ dB_t+w\sigma(\xi_t)dt
\end{align*}
with $\xi_0=x$ under $\PP^x$. By Remark~5.2.22 in~\cite{MR1121940}, the pair $\left((\xi_t)_{t\geq0},(\PP^x)_{x\in\cR}\right)$ is a Markov process with generator $L_{(\sigma,w)}$. We apply   the Doss-Sussmann technique to find an explicit expression for $(\xi_t)_{t\geq0}$. So, let $(B_t)_{t\geq0}$ be a standard $1$-dimensional Brownian motion with respect to some some probability measure $\PP$. We write $\E[\cdot]$ for the expectation under $\PP$ and  $\E^x[\cdot]$ for the one under $\PP^x$.  Let $g_\sigma$ be as in the statement of Proposition~\ref{prop:CB}. Then, by the It\^{o} formula for the Stratonovich integral
\begin{align*}
g_\sigma(B_t+wt,x)=x+\int_0^t \sigma\big( g_\sigma(B_s+ws,x) \big)\circ dB_s+\int_0^t w\sigma\big( g_\sigma(B_s+ws,x) \big)ds.
\end{align*}
Hence, for every $x\in\cR$,
\begin{align*}
\text{Law}\big( (g_\sigma(B_t+wt,x))_{t\geq0},\PP \big)=\text{Law}\big( (\xi_t)_{t\geq0},\PP^x \big).
\end{align*}
In view of Corollary~\ref{cor:subordinate case} and Example~\ref{example:with drift}, there is a nonnegative random variable $\aA$ (constructed from $k$ and $\gamma$ as in~\eqref{eq:new RV}) which is independent of $(B_t)_{t\geq0}$ and such that 
\begin{align*}
u(t,x)&=\E\left[ u_0\left(g_\sigma\left(B_{\aA t^{\theta/\gamma}}+w\aA t^{\theta/\gamma},x \right)  \right) e^{c\aA t^{\theta/\gamma}} \right]
\end{align*}
solves the evolution equation~\eqref{eq:EvEq-CB}. Note that  $\left( B_{\aA t^{\theta/\gamma}}+w\aA t^{\theta/\gamma}  \right)_{t\geq0}$, conditionally on $\aA$, is a Gaussian process with mean $w\aA t^{\theta/\gamma} $ and variance $\aA t^{\theta/\gamma}$. The process $\left( \sqrt{\aA} B_{t^{\theta/\gamma}}+w\aA t^{\theta/\gamma}  \right)_{t\geq0}$ and, if $H:=\frac{\theta}{2\gamma}\in(0,1)$, the process $\left(\sqrt{\aA} B^H_{t}+w\aA t^{\theta/\gamma}  \right)_{t\geq0}$ have the same conditional law, where $(B^H_t)_{t\geq0}$ is a $1$-dimensional fractional Brownian motion independent of $\aA$.  Hence Feynman-Kac formula~\eqref{eq:FKF-CB+drift} is shown. 
Further, we have
\begin{align*}
&u(t,x)=\E\left[ u_0\left(g_\sigma\left(\sqrt{\aA}B_{t^{\theta/\gamma}}+w\aA t^{\theta/\gamma},x \right)  \right) e^{c\aA t^{\theta/\gamma}} \right]\\
&
=\int_0^\infty\int_\cR u_0\big(g_\sigma(\sqrt{a}z+wat^{\theta/\gamma},x) \big)e^{cat^{\theta/\gamma}}(2\pi t^{\theta/\gamma})^{-1/2}\exp\left( -\frac{z^2}{2t^{\theta/\gamma}} \right)\,dz\,\mathcal{P}_\aA(da)\\
&
=\int_0^\infty\int_\cR u_0\big(g_\sigma(\sqrt{a}y,x) \big)e^{at^{\theta/\gamma}\left(c-w^2/2 \right)+w\sqrt{a}y}(2\pi t^{\theta/\gamma})^{-1/2}\exp\left( -\frac{y^2}{2t^{\theta/\gamma}} \right)\,dy\,\mathcal{P}_\aA(da)\\
&
=\E\left[ u_0\left(g_\sigma\left(\sqrt{\aA}B_{t^{\theta/\gamma}},x \right)  \right) e^{\aA t^{\theta/\gamma}\left(c-\frac{w^2}{2}  \right)+w\sqrt{\aA}B_{t^{\theta/\gamma}}} \right].
\end{align*}
 Hence Feynman-Kac formula~\eqref{eq:FKF-CB-drift} is shown.

\end{proof}


\section*{Acknowlwdgements}
Yana Kinderknecht (Butko) thanks Ren\'e Schilling for a fruitful discussion and interesting references.



\begin{thebibliography}{10}

\bibitem{MR4098657}
G.~Ascione, Y.~Mishura, and E.~Pirozzi.
\newblock Time-changed fractional {O}rnstein-{U}hlenbeck process.
\newblock {\em Fract. Calc. Appl. Anal.}, 23(2):450--483, 2020.

\bibitem{MR1874479}
B.~Baeumer and M.~M. Meerschaert.
\newblock Stochastic solutions for fractional {C}auchy problems.
\newblock {\em Fract. Calc. Appl. Anal.}, 4(4):481--500, 2001.

\bibitem{MR3059867}
S.~S. Bayin.
\newblock Time fractional {S}chr\"{o}dinger equation: {F}ox's {H}-functions and
  the effective potential.
\newblock {\em J. Math. Phys.}, 54(1):012103, 18, 2013.

\bibitem{MR4223052}
E.~Bazhlekova.
\newblock Completely monotone multinomial {M}ittag-{L}effler type functions and
  diffusion equations with multiple time-derivatives.
\newblock {\em Fract. Calc. Appl. Anal.}, 24(1):88--111, 2021.

\bibitem{bebu1}
{Bender, Christian, and Yana A. Butko}.
\newblock {Stochastic Solutions of Generalized time-fractional Evolution
  Equations}.
\newblock {\em {Frac. Calc. Appl. Anal.}}, 25(2), {2022}. DOI: 10.1007/s13540-022-00025-3.

\bibitem{MR30151}
S.~Bochner.
\newblock Diffusion equation and stochastic processes.
\newblock {\em Proc. Nat. Acad. Sci. U.S.A.}, 35:368--370, 1949.

\bibitem{MR3280006}
A.~G. Cherstvy, A.~V. Chechkin, and R.~Metzler.
\newblock Ageing and confinement in non-ergodic heterogeneous diffusion
  processes.
\newblock {\em J. Phys. A}, 47(48):485002, 18, 2014.

\bibitem{PhysRevLett.113.098302}
M.~V. Chubynsky and G.~W. Slater.
\newblock Diffusing diffusivity: A model for anomalous, yet {B}rownian,
  diffusion.
\newblock {\em Phys. Rev. Lett.}, 113:098302, Aug 2014.

\bibitem{MR1329992}
K.~L. Chung and Z.~X. Zhao.
\newblock {\em From {B}rownian motion to {S}chr\"{o}dinger's equation}, volume
  312 of {\em Grundlehren der Mathematischen Wissenschaften [Fundamental
  Principles of Mathematical Sciences]}.
\newblock Springer-Verlag, Berlin, 1995.

\bibitem{doi:10.1063/5.0029913}
J.~L. da~Silva and M.~Erraoui.
\newblock Singularity of generalized grey {B}rownian motion and time-changed
  {B}rownian motion.
\newblock {\em AIP Conference Proceedings}, 2286(1):020002, 2020.

\bibitem{MR1772266}
M.~Demuth and J.~A. van Casteren.
\newblock {\em Stochastic spectral theory for selfadjoint {F}eller operators. A functional integration approach}.
\newblock Probability and its Applications. Birkh\"{a}user Verlag, Basel, 2000.

\bibitem{DOSSANTOS2021110634}
M.~A. {dos Santos} and L.~{Menon Junior}.
\newblock Random diffusivity models for scaled {B}rownian motion.
\newblock {\em Chaos, Solitons and Fractals}, 144:110634, 2021.

\bibitem{MR574173}
H.~Doss.
\newblock Sur une r\'{e}solution stochastique de l'\'{e}quation de
  {S}chr\"{o}dinger \`a coefficients analytiques.
\newblock {\em Comm. Math. Phys.}, 73(3):247--264, 1980.

\bibitem{MR2754894}
H.~Doss.
\newblock On a probabilistic approach to the {S}chr\"{o}dinger equation with a
  time-dependent potential.
\newblock {\em J. Funct. Anal.}, 260(6):1824--1835, 2011.

\bibitem{MR3903618}
M.~D'Ovidio, S.~Vitali, V.~Sposini, O.~Sliusarenko, P.~Paradisi, G.~Castellani,
  and G.~Pagnini.
\newblock Centre-of-mass like superposition of {O}rnstein-{U}hlenbeck
  processes: a pathway to non-autonomous stochastic differential equations and
  to fractional diffusion.
\newblock {\em Fract. Calc. Appl. Anal.}, 21(5):1420--1435, 2018.

\bibitem{MR3417088}
J.~L.~A. Dubbeldam, Z.~Tomovski, and T.~Sandev.
\newblock Space-time fractional {S}chr\"{o}dinger equation with composite time
  fractional derivative.
\newblock {\em Fract. Calc. Appl. Anal.}, 18(5):1179--1200, 2015.

\bibitem{MR1721989}
K.-J. Engel and R.~Nagel.
\newblock {\em One-parameter semigroups for linear evolution equations}, volume
  194 of {\em Graduate Texts in Mathematics}.
\newblock Springer-Verlag, New York, 2000.

\bibitem{MR4174393}
P.~K. Friz and M.~Hairer.
\newblock {\em A course on rough paths.  With an introduction to regularity structures}.
\newblock Universitext. Springer, Cham, 2020. 

\bibitem{MR1467147}
R.~Gorenflo and Y.~Luchko.
\newblock Operational method for solving generalized {A}bel integral equation
  of second kind.
\newblock {\em Integral Transform. Spec. Funct.}, 5(1-2):47--58, 1997.

\bibitem{MR3609236}
P.~a. G\'{o}rka, H.~Prado, and J.~Trujillo.
\newblock The time fractional {S}chr\"{o}dinger equation on {H}ilbert space.
\newblock {\em Integral Equations Operator Theory}, 87(1):1--14, 2017.

\bibitem{MR2244037}
M.~Haase.
\newblock {\em The functional calculus for sectorial operators}, volume 169 of
  {\em Operator Theory: Advances and Applications}.
\newblock Birkh\"{a}user Verlag, Basel, 2006.

\bibitem{ISem21}
M.~Haase.
\newblock {\em Lectures on {F}unctional {C}alculus - 21st {I}nternational
  {I}nternet {S}eminar}.
\newblock
  https://www.math.uni-kiel.de/isem21/en/course/phase1/isem21-lectures-on-functional-calculus,
  2018.

\bibitem{MR1366608}
S.~B. Hadid and Y.~F. Luchko.
\newblock An operational method for solving fractional differential equations
  of an arbitrary real order.
\newblock {\em PanAmer. Math. J.}, 6(1):57--73, 1996.

\bibitem{MR1795612}
Y.~Hu and G.~Kallianpur.
\newblock Schr\"{o}dinger equations with fractional {L}aplacians.
\newblock {\em Appl. Math. Optim.}, 42(3):281--290, 2000.

\bibitem{MR3965362}
A.~Iomin.
\newblock Fractional time quantum mechanics.
\newblock In {\em Handbook of fractional calculus with applications. {V}ol. 5},
  pages 299--315. De Gruyter, Berlin, 2019.

\bibitem{Jain2017}
R.~Jain and K.~L. Sebastian.
\newblock Diffusing diffusivity: a new derivation and comparison with
  simulations.
\newblock {\em Journal of Chemical Sciences}, 129(7):929--937, Jul 2017.

\bibitem{MR1121940}
I.~Karatzas and S.~E. Shreve.
\newblock {\em Brownian motion and stochastic calculus}, volume 113 of {\em
  Graduate Texts in Mathematics}.
\newblock Springer-Verlag, New York, second edition, 1991.

\bibitem{MR3389585}
V.~Knopova.
\newblock On the {F}eynman-{K}ac semigroup for some {M}arkov processes.
\newblock {\em Mod. Stoch. Theory Appl.}, 2(2):107--129, 2015.

\bibitem{kolokoltsov2017chronological}
V.~Kolokoltsov.
\newblock Chronological operator-valued {F}eynman-{K}ac formulae for
  generalized fractional evolutions, 2017.

\bibitem{MR2766141}
V.~N. Kolokoltsov.
\newblock Generalized continuous-time random walks, subordination by hitting
  times, and fractional dynamics.
\newblock {\em Teor. Veroyatn. Primen.}, 53(4):684--703, 2008.

\bibitem{MR3987876}
V.~N. Kolokoltsov.
\newblock The probabilistic point of view on the generalized fractional partial
  differential equations.
\newblock {\em Fract. Calc. Appl. Anal.}, 22(3):543--600, 2019.

\bibitem{MR3613319}
M.~Kwa\'{s}nicki.
\newblock Ten equivalent definitions of the fractional {L}aplace operator.
\newblock {\em Fract. Calc. Appl. Anal.}, 20(1):7--51, 2017.

\bibitem{MR3821542}
N.~Laskin.
\newblock {\em Fractional quantum mechanics}.
\newblock World Scientific Publishing Co. Pte. Ltd., Hackensack, NJ, 2018.

\bibitem{MR0125601}
J.~V. Linnik.
\newblock {\em { Razlozheniya veroyatnostnykh zakonov}}.
\newblock Izdat. Leningrad. Univ., Leningrad, 1960.

\bibitem{MR3012216}
A.~Lunardi.
\newblock {\em Analytic semigroups and optimal regularity in parabolic
  problems}.
\newblock Modern Birkh\"{a}user Classics. Birkh\"{a}user/Springer Basel AG,
  Basel, 1995.

\bibitem{MR3373947}
M.~Magdziarz and R.~L. Schilling.
\newblock Asymptotic properties of {B}rownian motion delayed by inverse
  subordinators.
\newblock {\em Proc. Amer. Math. Soc.}, 143(10):4485--4501, 2015.

\bibitem{MR2592742}
F.~Mainardi, A.~Mura, and G.~Pagnini.
\newblock The {$M$}-{W}right function in time-fractional diffusion processes: a
  tutorial survey, {A}rt. {ID} 104505, 29 pages.
\newblock {\em Int. J. Differ. Equ.}, 2010.

\bibitem{MR2074812}
M.~M. Meerschaert and H.-P. Scheffler.
\newblock Limit theorems for continuous-time random walks with infinite mean
  waiting times.
\newblock {\em J. Appl. Probab.}, 41(3):623--638, 2004.

\bibitem{MR2442372}
M.~M. Meerschaert and H.-P. Scheffler.
\newblock Triangular array limits for continuous time random walks.
\newblock {\em Stochastic Process. Appl.}, 118(9):1606--1633, 2008.

\bibitem{C4CP03465A}
R.~Metzler, J.-H. Jeon, A.~G. Cherstvy, and E.~Barkai.
\newblock Anomalous diffusion models and their properties: non-stationarity{,}
  non-ergodicity{,} and ageing at the centenary of single particle tracking.
\newblock {\em Phys. Chem. Chem. Phys.}, 16:24128--24164, 2014.

\bibitem{MR1809268}
R.~Metzler and J.~Klafter.
\newblock The random walk's guide to anomalous diffusion: a fractional dynamics
  approach.
\newblock {\em Phys. Rep.}, 339(1):77, 2000.

\bibitem{MR2501791}
A.~Mura and F.~Mainardi.
\newblock A class of self-similar stochastic processes with stationary
  increments to model anomalous diffusion in physics.
\newblock {\em Integral Transforms Spec. Funct.}, 20(3-4):185--198, 2009.

\bibitem{MR2430462}
A.~Mura and G.~Pagnini.
\newblock Characterizations and simulations of a class of stochastic processes
  to model anomalous diffusion.
\newblock {\em J. Phys. A}, 41(28):285003, 22, 2008.

\bibitem{MR2588003}
A.~Mura, M.~S. Taqqu, and F.~Mainardi.
\newblock Non-{M}arkovian diffusion equations and processes: analysis and
  simulations.
\newblock {\em Phys. A}, 387(21):5033--5064, 2008.

\bibitem{MR3084964}
B.~N. Narahari~Achar, B.~T. Yale, and J.~W. Hanneken.
\newblock Time fractional {S}chrodinger equation revisited.
\newblock {\em Adv. Math. Phys.},  Art. ID 290216, 11 pages, 2013.

\bibitem{MR3586912}
G.~Pagnini.
\newblock Fractional kinetics in random/complex media.
\newblock In {\em Handbook of fractional calculus with applications. {V}ol. 5},
  pages 183--205. De Gruyter, Berlin, 2019.

\bibitem{MR3513003}
G.~Pagnini and P.~Paradisi.
\newblock A stochastic solution with {G}aussian stationary increments of the
  symmetric space-time fractional diffusion equation.
\newblock {\em Fract. Calc. Appl. Anal.}, 19(2):408--440, 2016.

\bibitem{MR710486}
A.~Pazy.
\newblock {\em Semigroups of linear operators and applications to partial
  differential equations}, volume~44 of {\em Applied Mathematical Sciences}.
\newblock Springer-Verlag, New York, 1983.

\bibitem{MR3919012}
I.~Petreska, A.~S.~M. de~Castro, T.~Sandev, and E.~K. Lenzi.
\newblock The time-dependent {S}chr\"{o}dinger equation in three dimensions
  under geometric constraints.
\newblock {\em J. Math. Phys.}, 60(3):032101, 8, 2019.

\bibitem{MR50797}
R.~S. Phillips.
\newblock On the generation of semigroups of linear operators.
\newblock {\em Pacific J. Math.}, 2:343--369, 1952.

\bibitem{MR2978140}
R.~L. Schilling, R.~Song, and Z.~Vondra\v{c}ek.
\newblock {\em Bernstein functions. Theory and applications}, volume~37 of {\em De Gruyter Studies in
  Mathematics}.
\newblock Walter de Gruyter \& Co., Berlin, second edition, 2012.

\bibitem{MR3916448}
O.~Y. Sliusarenko, S.~Vitali, V.~Sposini, P.~Paradisi, A.~Chechkin,
  G.~Castellani, and G.~Pagnini.
\newblock Finite-energy {L}\'{e}vy-type motion through heterogeneous ensemble
  of {B}rownian particles.
\newblock {\em J. Phys. A}, 52(9):095601, 27, 2019.

\bibitem{Sposini_2018}
V.~Sposini, A.~V. Chechkin, F.~Seno, G.~Pagnini, and R.~Metzler.
\newblock Random diffusivity from stochastic equations: comparison of two
  models for {B}rownian yet non-{G}aussian diffusion.
\newblock {\em New Journal of Physics}, 20(4):043044, apr 2018.

\bibitem{Wang_2020}
W.~Wang, A.~G. Cherstvy, A.~V. Chechkin, S.~Thapa, F.~Seno, X.~Liu, and
  R.~Metzler.
\newblock Fractional brownian motion with random diffusivity: emerging residual
  nonergodicity below the correlation time.
\newblock {\em Journal of Physics A: Mathematical and Theoretical},
  53(47):474001, nov 2020.

\bibitem{MR0399953}
U.~Westphal.
\newblock Gebrochene {P}otenzen abgeschlossener {O}peratoren, definiert mit
  {H}ilfe gebrochener {D}ifferenzen.
\newblock In {\em Linear operators and approximation, {II} ({P}roc. {C}onf.,
  {O}berwolfach {M}arh. {R}es. {I}nst., {O}berwolfach, 1974)}, pages 23--27.
  Internat. Ser. Numer. Math., Vol. 25, 1974.

\end{thebibliography}

\end{document}